\documentclass[11pt,letterpaper]{amsart}

\usepackage[utf8]{inputenc}
\usepackage{gauss}
\usepackage[all]{xy}
\usepackage{tikz}
\usepackage{tikz-cd}
\usetikzlibrary{patterns}
\usepackage{csquotes}
\usepackage[kerning,spacing]{microtype}
\usepackage{amssymb,latexsym,amsmath,amscd,graphicx,amsthm}
\usepackage{hyperref}
\usepackage{colortbl}
\usepackage{marginnote}
\usepackage[nameinlink]{cleveref}
\usepackage{mathscinet}

\microtypecontext{spacing=nonfrench}

\newtheorem{mthm}{Theorem}
\newtheorem{thm}{Theorem}[section]

\newtheorem{lem}[thm]{Lemma}

\newtheorem{prop}[thm]{Proposition}

\theoremstyle{definition}
\newtheorem{rem}[thm]{Remark}
\newtheorem*{rem*}{Remark}

\newtheorem{defn}[thm]{Definition}

\newtheorem{notn}[thm]{Notation}

\newtheorem{example}[thm]{Example}

\def\cC{\mathcal{C}}

\def\cL{\ensuremath{\mathcal{L}}}

\def\cO{\mathcal{O}}

\def\cX{\mathcal{X}}

\def\cm{\mathfrak{m}}

\def\bA{{\mathbb A}}

\def\bC{{\mathbb C}}

\def\bG{{\mathbb G}}

\def\bK{{\mathbb K}}

\def\bN{{\mathbb N}}
\def\bP{{\mathbb P}}
\def\bQ{{\mathbb Q}}
\def\bR{{\mathbb R}}

\def\bZ{{\mathbb Z}}

\def\sxymat{\xymatrix@C=1.5ex@R=2ex}
\def\ubK{\underline{\bK}}

\newcommand{\lr}[1]{\langle{#1}\rangle}

\DeclareMathOperator{\Span}{Span}

\DeclareMathOperator{\Spec}{Spec}
\DeclareMathOperator{\Proj}{Proj}

\DeclareMathOperator{\sst}{sst}
\DeclareMathOperator{\ST}{ST}

\DeclareMathOperator{\Sym}{Sym}
\DeclareMathOperator{\Hom}{Hom}

\DeclareMathOperator{\red}{red}

\DeclareMathOperator{\id}{id}
\DeclareMathOperator{\ord}{ord}

\DeclareMathOperator{\Bl}{Bl}
\DeclareMathOperator{\Star}{Star}

\DeclareMathOperator{\supp}{supp}

\DeclareMathOperator{\IC}{IC}

\DeclareMathOperator{\Stab}{Stab}

\DeclareMathOperator{\conv}{conv}

\DeclareMathOperator{\wt}{wt}

\DeclareMathOperator{\class}{cl}
\DeclareMathOperator{\cell}{cell}
\newcommand{\cupdot}{\mathbin{\mathaccent\cdot\cup}}

\def\tensor{\otimes}
\def\tto{\twoheadrightarrow}

\def\oST{\overline{\ST}}

\def\mmod{/\!/}

\newcommand{\map}[1]{\stackrel{#1}{\longrightarrow}}

\newcommand{\un}[1]{\underline{#1}}

\def\dar{\ar@<-0.5ex>[r]\ar@<0.5ex>[r]}

\def\tto{\twoheadrightarrow}

\definecolor{lightgray}{gray}{0.9}
\definecolor{darkgreen}{RGB}{0,100,0}
\definecolor{darkblue}{RGB}{0,0,180}
\definecolor{darkred}{RGB}{210,0,0}
\definecolor{gb}{RGB}{0,100,180}

\numberwithin{equation}{section}

\title{On the moment measure conjecture}%
\author[J.~Heinloth]{Jochen Heinloth}
\address{Universität Duisburg-Essen, Essen, Germany}
\email{jochen.heinloth@uni-due.de}
\author[X.~Zhang]{Xucheng Zhang}
\address{Yau Mathematical Sciences Center, Tsinghua University, Beijing, China}
\email{zhangxucheng@mail.tsinghua.edu.cn}

\begin{document}
\begin{abstract}
	The moment measure conjecture of Bia\l ynicki-Birula and Sommese gives a combinatorial characterization of all open substacks of a global quotient stack for a torus action on a normal projective variety that admit a proper good moduli space, in other words it characterizes the invariant open subvarieties that admit a proper quotient. In this article we prove the conjecture for actions on smooth varieties. This gives an instance of stability conditions defined by cohomological invariants that are not given by Chern classes of line bundles.  
\end{abstract}

\maketitle

\section{Introduction}

The existence result for moduli spaces of algebraic stacks \cite{ahlh} gives a new motivation for the problem to find methods, that single out those open substacks of a given algebraic stack that admit proper moduli spaces. One instance where a conjectural answer to this problem is known is the case of global quotient stacks $[X/T]$ that arise from a torus acting on a projective variety, as  this is the content of the moment measure conjecture of  Bia\l ynicki-Birula and Sommese. This case in particular covers examples of quotient spaces that are proper but not projective and therefore cannot be obtained from a GIT construction. In this article we prove a version of this conjecture using a cohomological interpretation of the combinatorial data of moment measures. In contrast to GIT, in which the numerical criterion depends on the Chern class of a line bundle, i.e., an element of the second cohomology of the stack, our interpretation of moment measures is given in terms of (co)homology classes in the degree corresponding to the top cohomological degree of the moduli space. In short, we observe that any open subset of a smooth algebraic stack that admits a proper coarse moduli space is uniquely determined by the linear algebra datum of a one dimensional quotient defined by the top cohomology group of the moduli space and show that this description can be used effectively to classify such substacks in the example of torus quotients.

The first version of the moment measure conjecture \cite[Conjecture (1.3)]{BBS-conjecture} of Bia\l ynicki-Birula and Sommese proposes an explicit combinatorial characterization of the invariant open subsets of a normal projective variety $X$ equipped with a torus action, that admit a proper geometric quotient. Bia\l ynicki-Birula and Sommese also formulate a version for good quotients \cite[Conjecture (1.4)]{BBS-conjecture} (see also \cite{BBQuotientsSurvey} for a survey on these conjectures). The sufficiency of the conditions was known to the authors when they formulated the conjecture, so the main point is to show that the combinatorial criterion is indeed necessary. 

Before formulating the conjecture, the authors had proved the result for $1$-dimen\-sion\-al tori \cite[Main Theorem]{BBS-QuotientsRk1} 
and for geometric quotients of $2$-dimensional tori acting meromorphically on a compact Kähler manifold \cite[(3.1) Main Theorem]{BBS-rank2}. Later Bia\l ynicki-Birula and Świ\polhk{e}cicka \cite{BBSwiecickaReciepe},  Bia\l ynicki-Birula \cite{BBMomentMeasureProjective} and Hausen \cite{HausenCompleteAffine} proved that the conjecture holds for actions on affine spaces, projective spaces, and affine varieties. 

To formulate our result let $T=\bG_m^r$ be a torus over an algebraically closed field, acting on a smooth proper variety $X$. We assume that there exists a $T$-linearized line bundle $\cL$ on $X$ that restricts to an ample bundle on all orbit closures of closed points (\Cref{def:positiveL}, \Cref{rem:projective_orbits}). 

From this datum Bia\l ynicki-Birula and Sommese defined a cell complex $\cC(X)$, whose $i$-dimensional cells record which {collections of} fixed point {components} can appear in the closure of $i$-dimensional $T$-orbits of closed points (see \Cref{sec:moment_complex}). Their conjecture states that the invariant open subsets $U \subseteq X$ which admit proper geometric and proper good quotients can be described combinatorially in terms of functions $m\colon \cC(X) \to \{0,1\}$ called (geometric) moment measures (see \Cref{def:MomentMeasure}). 
A moment measure defines an invariant open subset $U\subseteq X$ consisting of those points $x$ such that the orbit closure $\overline{T.x}$ contains an orbit on whose cell the moment measure takes value $1$.  

One reason why this conjecture is surprising, is that the set of points of $X$ that define the same cell, i.e., the set of points whose orbit closure intersects the same collection of fixed point {components}, can be disconnected and it is unclear why stability conditions do not depend on the connected components of these subsets. A second point is that although the definition of the moment complex $\cC(X)$ is built from data coming from orbit closures, one of the combinatorial conditions includes cell decompositions that do not have to arise geometrically from a degeneration of an orbit closure. 

The main result of the article is the following theorem.

\begin{mthm}[Conjecture of Bia\l ynicki-Birula--Sommese]\label{thm:GeometricQuotients}
	Let $X$ be a proper smooth variety over an algebraically closed field equipped with a torus action $T\times X \to X$ that admits a line bundle $\cL$ that is ample on $\bG_m$-orbit closures of closed points. 
	
	A $T$-invariant open subscheme $U\subseteq X$ admits a proper geometric quotient if and only if it is defined by a geometric moment measure on the moment complex $\cC(X)$.
	
	A $T$-invariant open subscheme $U\subseteq X$ admits a proper good quotient if and only if it is defined by a moment measure on the moment complex $\cC(X)$.
\end{mthm}
\begin{rem*} Our positivity assumption on the line bundle is more general than in the original formulation of the conjecture. This is useful as in the case of moduli problems it is often easier to prove positivity on the rational lines defined by one parameter degenerations than to check a global condition like ampleness. For example, we will apply this in joint work of L. Halle, K. Hulek and the first author, to explain how  the construction of \cite{GHH-HilbertScheme} works for non-projective degenerations.
\end{rem*}
\begin{rem*}
The original formulation of the conjecture allows for normal instead of smooth varieties, probably because this would facilitate an inductive approach to the conjecture, which we avoid. As moment complexes behave well under blow-ups, we expect that in characteristic $0$ existence of equivariant resolution could be used to reduce the normal case to the case of a smooth variety.  
\end{rem*}

Our strategy to construct a moment measure from an open subset $U$ is the following: The closure of any $T$-orbit $\overline{T.x} \subseteq X$ defines an equivariant cohomological cycle class $$\class_{[X/T]}([\overline{T.x}/T]) \in H^{2 (\dim [X/T]-\dim [T.x/T])}([X/T],\bK),$$
where cohomology can either be taken as singular cohomology with rational coefficients $\bK=\bQ$ if $k\subseteq \bC$ or as étale cohomology with coefficients in $\bK=\bQ_\ell$ for some prime $\ell\neq \text{char}(k)$.

If $U\subseteq X$ admits a proper geometric quotient $U\mmod T$, the cohomology group in top degree $H^{2\dim(U\mmod T)}(U\mmod T,\bK)$ is 1-dimensional and in this case the cycle class of an orbit closure of maximal dimension will restrict to a non-zero class in $$H^{2\dim(U\mmod T)}([U/T],\bK)\cong H^{2\dim(U\mmod T)}(U\mmod T,\bK)$$ if and only if the orbit closure intersects $U$. To obtain a moment measure from this datum, we observe that we can alternatively compute the cycle class in the cohomology of the quotient stack $H^{2\dim(U\mmod T)}([X/T],\bK)$ by equivariant localization to the fixed points of the torus action. In this description it turns out that the cycle class of an orbit closure is uniquely  determined by its cell (\Cref{prop:CycleClassByCell}). The corresponding moment measure can thus be defined as the function selecting those cells that define a cohomology class in $H^*([X/T],\bK)$ that restricts to a generator of the top cohomology of the moduli space.

For good quotients a similar strategy works, i.e., points of $U$ can still be detected as those points whose cycle class does not restrict to $0$ in $H^*([U/T],\bK)$, but in this case we have to consider cycle classes of orbits of different dimensions, that occur in different cohomological degrees. Moreover in this case, due to positive dimensional stabilizers of points in $U$ the cohomology group in degree $2\dim [U/T]$ is no longer one dimensional and thus not isomorphic to the corresponding cohomology group of the good quotient. This complicates the argument for good quotients.

\subsection*{Acknowledgments} The approach to the problem was found during the special semester on Moduli and Algebraic Cycles at the Mittag-Leffler Institute. JH would like to thank the organizers John Cristian Ottem, Dan Petersen and David Rydh for the inspiring semester and the Mittag Leffler Institute for the excellent environment. JH would also like to thank Daniel Greb for discussions about the conjecture and Marc Levine for many discussions on equivariant geometry and cycles. We would like to thank our referees for very helpful comments and suggestions. Part of this work was supported by the DFG-Research Training Group 2553. 
The authors have no conflict of interest to declare. The article has no associated data.

\subsection*{Notation and conventions} We will work over an algebraically closed field $k$. 
  We fix a torus $T=\bG_m^r$. We write $X^*(T)=\Hom(T,\bG_m) \cong \bZ^r$ for the character lattice of $T$, $X_*(T):=\Hom(\bG_m,T)$ for the cocharacter lattice and $$\lr{\quad,\quad}\colon X^*(T) \times X_*(T) \to \bZ$$ for the natural pairing of characters and cocharacters. We will write $X^*(T)_\bR:=X^*(T)\tensor_{\bZ}\bR$ and $X^*(T)_\bQ:=X^*(T)\tensor_{\bZ}\bQ$ for the real and rational vector spaces defined by the character lattice.
 
 We will write $H^*(\un{\quad},\bK)$ to denote either étale cohomology with $\bQ_\ell$-coefficients with $\ell\neq \text{char}(k)$ if $k\not\subseteq \bC$ or  singular cohomology with $\bQ$-coefficients if $k\subseteq \bC$. 
 
 Given a $T$-action on a proper variety $X$, a cocharacter $\lambda\colon \bG_m \to T$ and $x\in X(k)$ a closed point we write $\overline{\lambda.x}\colon \bP^1 \to X$ for the completion of the orbit map $\lambda.x\colon \bG_m \to X$ given by $t\mapsto \lambda(t).x$. We write $\lim_{t\to 0} \lambda(t).x$ (resp. $\lim_{t\to \infty} \lambda(t).x$) for the point defined by evaluation of $\overline{\lambda.x}$ at $0\in \bP^1$ (resp. $\infty\in \bP^1$).
 
\section{Recollection on moment measures}\label{sec:momentmeasure}	

In this section we recall the definition of moment measures and the moment complex of a variety with a torus action as defined by Bia\l ynicki-Birula and Sommese. As our main result allows for a slightly weaker hypothesis than projectivity of the underlying variety, we include proofs for the basic results on the structure of moment complexes that allow for this slightly weaker assumption. The arguments needed are however close to the original proofs and sometimes reduce to the projective situation.  

Throughout we will consider an action of the $r$-dimensional  torus $T=\bG_m^r$ on a proper, connected, normal variety $X$. 

\subsection{Recollection on the geometry of torus actions}
The closed subscheme of $T$-fixed points is denoted by $X^T\subseteq X$. We denote the set of connected components of $X^T$ by $\pi_0(X^T)$ and the decomposition into {connected components} by
$$X^T = \coprod_{i\in \pi_0(X^T)} F_i.$$

As closed orbits in $X$ are proper and the only proper quotients of $T$ are points, the only closed orbits of $T$ in $X$ are $T$-fixed points. If $X$ is smooth, the scheme of fixed points $X^T$ is smooth as well (\cite[Theorem 2.1]{BBactions}), so that in this case the connected components $F_i$ are smooth and irreducible.

By Sumihiro's theorem \cite[Corollary 2]{SumihiroEquivariantCompletion} any point $x\in X$ is contained in an affine open $T$-{invariant} neighborhood $U\subseteq X$, so we can find a closed $T$-equivariant embedding $U\hookrightarrow \bA^N$ into an affine space, equipped with a linear $T$-action. In particular the fixed points in $U$ are then the intersection of $U$ with the affine subspace $\bA^{n_0}\subseteq\bA^N$ on which $T$ acts trivially.  

This description allows to reduce computations of orbit closures to the case of linear actions on affine spaces. In particular any specialization of orbits can be realized by limits of cocharacters, i.e., if $x\in X(k)$ is a closed point and $T.x_1\subseteq \overline{T.x}$ is an orbit contained in the orbit closure, then there exists a cocharacter $\lambda\in X_*(T)$ such that $\lim_{t\to 0} \lambda(t).x\in T.x_1$. 

\subsection{Definition of the moment complex}\label{sec:moment_complex}

For any closed point $x\in X(k)$ the {\em cell} of $x$ is the subset
$$\cell(x):= \{ i\in \pi_0(X^T) \,|\, \overline{T.x} \cap F_i \neq \varnothing\}.$$
We will consider this subset as a $\dim (T.x)$-dimensional cell of a combinatorial cell complex. We will give a geometric realization of this cell in \Cref{prop:geometric_cells} (1). 

\begin{rem}
It will sometimes be convenient to extend this definition to all points of $X$, so for any field extension $K/k$  and $x_K\in X(K)$ we again denote the orbit closure of $x_K$ in $X_K:=X\times_k \Spec(K)$ by $\overline{T.x_K}\subseteq X_K$ and define $$\cell(x_K):= \{ i\in \pi_0(X^T) \,|\, \overline{T.x_K} \cap F_{i,K} \neq \varnothing\}.$$ As the set of points $x\in X$ whose cell is a given subset $c\subseteq \pi_0(X^T)$ is constructible, any cell of a $K$-point also appears as a cell of a closed point.
\end{rem}
 
\begin{defn}[Moment complex \cite{BBS-conjecture}]
	For a normal proper variety $X$ equipped with a torus action $T\times X \to X$, the {\em moment complex} $$\cC(X):= \{ c \subseteq \pi_0(X^T) \mid c=\cell(x) \text{ for some }x\in X\}$$ is the partially ordered set, given by the subsets of $\pi_0(X^T)$ that arise as  cells of points of $X$. 
	
	The {\em generic cell} is the cell defined by the generic point of $X$.  
	
	The {\em boundary} of a cell $c$ consists of those subsets $c^\prime \subsetneq c$ that happen to be cells. 
\end{defn}
The motivation to consider the set $\cC(X)$ and to equip it with the structure of a cell complex, stems from projective varieties $X$ over the complex numbers, where one can identify the cells with the moment polytopes of the orbit closures. To get such a description over general base fields it is convenient to describe the moment polytopes algebraically. We will recall this in the next section and adapt the argument to cover our slightly more general setup.

\subsection{Geometric realization of the moment complex}

Given a line bundle $\cL$ on $X$, the $T$-action on $X$ lifts to a power of $\cL$, i.e., some power of $\cL$ admits a $T$-linearization, because we assumed $X$ to be normal. This choice is moreover unique up to a choice of a character of $T$, because the difference of two linearizations defines a linearization of the trivial line bundle. Let $\cL$ be a fixed $T$-linearized line bundle on $X$. We will assume the following positivity condition on $\cL$, that is automatic if $\cL$ is ample.
\begin{defn}\label{def:positiveL}
	A line bundle $\cL$ on a proper $T$-variety $X$ is called {\em ample on $\bG_m$-orbit closures of closed points} if for all cocharacters $\lambda \in X_*(T)$ and closed points $x\in X$ such that $x$ is not a $\lambda$-fixed point, the pull back of $\cL$ along the morphism $\overline{\lambda.x} \colon \bP^1 \to X$ defined by the completion of the orbit map $\lambda.x$ is ample.
\end{defn}
\begin{rem}
This positivity condition already appears naturally when considering families $X\to S$ with a fibrewise torus action, where ampleness on the fibers and ampleness on $X$ differ. To apply a GIT construction, one would need to use an ample bundle on $X$, which may not exist if $S$ is a non-projective variety, although in the Hilbert-Mumford criterion of GIT only restrictions to orbit closure appear. This situation appears for example when one tries to use inductive arguments to reduce problems on torus quotients to quotients by one-dimensional tori.
\end{rem}

For our computations it will be helpful to translate this positivity property in terms of weights on fixed points. 

For any $T$-fixed point $x \in X^T$ we denote by $$\wt(\cL,x)\in X^*(T)=\Hom(T,\bG_m)$$ the character of $T$ given by the action of $T$ on the fiber $\cL_x$ of $\cL$ at $x$. In this article we will use the sign convention taking the weights of the action on the fiber of the sheaf of sections $\cL$, which is the one used in \cite{ahlh}. In \cite{heinloth-stability} we used the opposite sign convention taking weights on the geometric realization $L=\Spec (\Sym^\bullet \cL^\vee)$. Let us spell this out concretely in the basic example that is the key for computations.	
\begin{example}\label{ex:weight-degree}
	Let $\bG_m \times \bA^1 \to \bA^1$ be the standard action given on points by $t.a:=ta$. On the ring of functions $k[x_1]$ of $\bA^1=\Spec k[x_1]$ this is the action for which $x_1$ has weight $-1$. The action extends to $\bP^1=\Proj(k[x_0,x_1])$, such that the homogeneous coordinates $x_0$ and $x_1$ have weight $0$ and $-1$ respectively. For the induced linearization of line bundle $\cO(d)$  corresponding to homogeneous sections of degree $d$ we have $$\wt(\cO(d),0)=0, \wt(\cO(d),\infty)=-d$$
	because the fibers of $\cO(d)$ at $0$ and $\infty$ are generated by the global sections $x_0^d$ and $x_1^d$ respectively.
	Thus $$\deg(\cO(d))= \wt(\cO(d),0)-\wt(\cO(d),\infty).$$
	As any other linearization of $\cO(d)$ differs from the standard one by a character, this shows that for any $\bG_m$-linearized line bundle $\cL$ on $\bP^1$ the degree is determined by the weights:
	$$\deg(\cL)= \wt(\cL,0)-\wt(\cL,\infty).$$
\end{example}
Returning to our general setup, we now assume that $\cL$ is a $T$-linearized line bundle on $X$ that is ample on $\bG_m$-orbit closures of closed points. To give a geometric interpretation of cells of points note that the weight $\wt(\cL,\un{\quad})$ is constant on the connected components $F_i$ of $X^T$ and thus defines a map
\begin{align*}
	\wt_{\cL} \colon \pi_0(X^T) &\to X^*(T)\\
	i &\mapsto \wt_{\cL}(i):= \wt(\cL, x) \text{ for any }x\in F_i.
\end{align*}
For complex projective varieties and Kähler manifolds Bia\l ynicki-Birula and Sommese proved the following geometric interpretation of cells, using that in complex geometry the convex hull $$ |\cell(x)|:= \conv(\wt_{\cL}(\cell(x))) \subseteq X^*(T)_\bR,$$ is the moment polytope of the orbit and the moment map identifies the topological quotient $\overline{T.x}/(S^1)^r$ with the moment polytope. We will recall an algebraic argument for this result below. 
We will refer to the polytope $|\cell(x)|\subseteq X^*(T)_\bR$ as the {\em geometric realization} of $\cell(x)$ and write $\mathring{|\cell(x)|}$ for the relative interior of the polytope, i.e. the complement of the geometric realization of the proper faces of $|\cell(x)|$.

\begin{prop}\label{prop:geometric_cells}
	Let $\cL$ be a $T$-linearized line bundle on $X$ that is ample on $\bG_m$-orbit closures of closed points.
	\begin{enumerate}
		\item\label{item:polytope} For any point $x\in X$ the characters $\{\wt_{\cL}(i) \mid i\in \cell(x) \}$ are the vertices of the convex polytope $|\cell(x)|\subseteq X^*(T)_{\bR}$ and the real dimension of the polytope is equal to the algebraic dimension of the orbit $T.x\subseteq X$. 
		\item\label{item:subdivision} Let $R$ be a DVR with fraction field $K$ and residue field $R/\cm$ and $x_R\in X(R)$ an $R$-valued point and let $Z_R:= \overline{T.x_R} \subseteq X \times_k \Spec R$ be the closure of the $T$-orbit of $x_R$. Let $Z_0=\cup \overline{T.x_j}$ be the maximal dimensional $T$-orbits of the special fiber $Z_0:=Z_R\times_R \Spec(R/\cm)$. Then the geometric realizations of $|\cell(x_j)|$ define a subdivision of $|\cell(x_K)|$.  
	\end{enumerate}
\end{prop}
\begin{proof}[Proof for ample line bundles]
Let us first recall the argument in the case that $\cL$ is ample: Replacing $\cL$ by a power if necessary, we may assume that $\cL$ is very ample, because $\wt_{\cL^n}=n\cdot \wt_{\cL}$ and thus this replacement does not affect statements (1) and (2).  Let $X\subseteq \bP(H^0(X,\cL))$ be the $T$-equivariant projective embedding defined by $\cL$, let $\Phi(\cL)=\{\chi_1,\dots,\chi_N\} \subseteq X^*(T)$ be the weights of $T$ on $H^0(X,\cL)$. The $T$-fixed points in $\bP(H^0(X,\cL))$ are exactly the projective subspaces defined by the weight-spaces $\bP(H^0(X,\cL)_\chi)\subseteq \bP(H^0(X,\cL))$ for some $\chi\in \Phi(\cL)$ and the weight of $\cL$ at a point $x_\chi$ in the weight-space $\bP(H^0(X,\cL)_{\chi})\subseteq \bP(H^0(X,\cL))$ is by construction $$\wt(\cL,x_\chi)=\chi.$$  

For $x\in X$ let $\supp(x)\subseteq \Phi(\cL)$ be the subset of those weights for which $x$ has a non-zero coordinate in the corresponding weight space.

For given $x\in X$ and any cocharacter $\lambda \in X_*(T)$ the limit $x_{\lambda,0}:=\lim_{t\to 0}\lambda(t).x$ exists in $X$ and $\supp(x_{\lambda,0}) \subseteq \supp(x)$ is the subset of those weights $\chi\in \supp(x)$, for which $\langle \chi,\lambda\rangle$ is minimal. Moreover, the limit point $x_{\lambda,0}$ is a fixed point for the $T$-action if and only if $\supp(x_{\lambda,0})=\{\chi\}$ is a single character. Conversely, a weight $\chi\in \supp(x)$ appears in this way if and only if there exists a cocharacter $\lambda\in X_*(T)$ for which $\lr{\chi,\lambda}$ is the minimal value of $\lambda$ on $\supp(x)$, i.e., if and only if $\chi$ is a vertex of the convex hull of $\supp(x)$. In this case $\chi = \wt(\cL,(x_{\lambda,0}))$ and in particular these weights are the vertices of the convex polytope $|\cell(x)| \subseteq X^*(T)_{\bR}$. This proves (\ref{item:polytope}) for projective $X$. 

For projective $X$ part (\ref{item:subdivision}) also follows from a direct computation:  Write $x_K=[x_\chi]_{\chi\in\Phi(\cL)}\in \bP(H^0(X,\cL))(K)$ for the projective coordinates of $x_K$ where $x_\chi$ is the coordinate of $x_K$ in the weight space corresponding to the character $\chi$. 

The points of the special fiber can be obtained as specializations of translates $x^\prime_{K^\prime}=t_{K^\prime}.x_K$ for some finite extension $R^\prime/R$ of DVR's with fraction field $K^\prime$  and some $t_{K^\prime}\in T(K^\prime)$. To determine the support of the specialization, let $\pi^\prime\in R^\prime$ be a uniformizing element and write $t_K^\prime$ as $\lambda(\pi^\prime)\cdot t_{R^\prime}$ for some $\lambda\in X_*(T)$ and $t_{R^\prime}\in T(R^\prime)$. Then the support of the specialization of $x^\prime_{K^\prime}=t_{K^\prime}.x_K$ in the special fiber consists of exactly those elements of $\supp(x_K)$ for which $\lr{\lambda,\chi}+|x_\chi|$ is minimal, where $|x_\chi|$ is the order of the coordinate $x_\chi$. Thus the possible supports that can be obtained in this way are defined by the linear parts of the maximal convex piecewise linear function $v_x\colon X^*(T)\to \bR$ that satisfies $v_x(\chi)\leq -|x_\chi|$. As these define a subdivision of $|\cell(x_K)|$ this proves (\ref{item:subdivision}) for projective $X$.
\end{proof}
We will reduce the general case to the projective situation. To do this, we need to provide analogues of the above polyhedral descriptions.
\begin{lem}\label{lem:orbitproj}
	Let $X$ be a proper normal variety with a torus action $T\times X \to X$ and $\cL$ a $T$-linearized line bundle on $X$ that is ample on $\bG_m$-orbit closures of closed points and $x\in X$ a closed point. Let $\{x_i\}_{i\in\cell(x)}$ be the closed orbits in $\overline{T.x}$.
	Then we have:
	\begin{enumerate}
		\item For any $i\in\cell(x)$ the set $$C_i:=\{ \lambda \in X_*(T) \mid  \lim_{t\to 0} \lambda(t).x =x_i\} \subseteq X_*(T)$$
		is the relative interior of a convex polyhedral cone $\overline{C}_i$.
		\item The union $X_*(T)= \cup_{i\in \cell(x)} \overline{C_i}$ defines a fan in $X_*(T)$ and
		the function 
		\begin{align*}
			\wt_{\cL,x}\colon X_*(T) &\to \bZ \\
			\lambda & \mapsto \lr{\wt(\cL,\lim_{t\to 0} \lambda(t).x),\lambda}
		\end{align*} defines a piecewise linear function that is strictly convex. 
		\item The weights $\chi_i:=\wt(\cL,x_i) \in X^*(T)$ for $i\in \cell(x)$ are the vertices of a convex polytope in $X^*(T)_\bR$.
	\end{enumerate}
\end{lem}
\begin{proof}
	Given $x\in X$ and a closed orbit $x_i\in \overline{T.x}$, as before there exists an affine open $T$-{invariant} neighborhood $U_i\subseteq X$ of $x_i$ by Sumihiro's theorem \cite[Corollary 2]{SumihiroEquivariantCompletion} and this contains $T.x$, as it intersects $\overline{T.x}$ non-trivially. Choosing a $T$-equivariant embedding $U_i\hookrightarrow \bA^N$ into an affine space equipped with a linear action of $T$ and noting that $T$-fixed points in $\bA^N$ are the weight zero subspace, we see that we may assume that under the embedding $U_i\to \bA^N$ the fixed point $x_i$ is mapped to $0\in \bA^N$. 
	
	Let us write $\Phi(\bA^N)\subseteq X^*(T)$ for the weights appearing in the $T$-representation $\bA^N$ and $(x_\chi)_{\chi \in \Phi(\bA^N)}$ for the weight decomposition of the image of $x$ in $\bA^N$. As before let $\supp(x)\subseteq \Phi(\bA^N)$ denote the set of weights for which  $x_\chi\neq 0$. By construction $0$ is in the closure of the orbit of $(x_\chi)_{\chi\in \Phi(\bA^N)}$, so there exists a cocharacter $\lambda\in X_*(T)$ such that $\lim_{t\to 0}\lambda(t).(x_\chi)=0$. The condition $\lim_{t\to 0}\lambda(t).(x_\chi)=0$ is equivalent to the condition that $\lr{\lambda,\chi}>0$ for all $\chi\in \supp(x)$. Thus 
	$$C_i:=\{ \lambda \in X_*(T) \mid  \lim_{t\to 0} \lambda(t).x =x_i\} \subseteq X_*(T)$$
	is the non-empty relative interior of a polyhedral cone, which proves (1).
	
	To prove (2) note that faces of the cone $\overline{C_i}$ are given by cocharacters $\lambda$ satisfying $\lr{\lambda,\chi}\geq 0$ for all $\chi\in \supp(x)$ and with equality for some $\chi$. For these $\lambda$ the limit $\lim_{t\to 0}\lambda(t).(x_\chi)=:x_\lambda$ exists in $\bA^N$ and is a point $x_\lambda \in \overline{T.x}\cap U_i$, as $x_\lambda$ specializes to $x_i \in U_i$. The point $x_\lambda$ is stabilized by a positive dimensional subgroup $\cap_{\chi\in \supp(x),\atop \lr{\lambda,\chi}=0} \ker(\chi)\subsetneq T$. Moreover all $T$-orbits in $\overline{T.x}\cap U_i\subseteq \bA^N$ arise in this way. 
	
	As $X$ is proper, we also know that for all cocharacters $\lambda \in X_*(T)$ the limit $\lim_{t\to 0} \lambda(t).x = x_\lambda$ exists in $X$ and the orbit closure $\overline{T.x_\lambda}$ contains some $T$-fixed point $x_i$. Thus $\overline{T.x_\lambda}$ defines a face of $\overline{C_i}$, so $X_*(T)= \cup_{i\in \cell(x)} \overline{C_i}$ defines a fan in $X_*(T)$.
	
	Finally the function $$\lambda \mapsto \lr{\wt(\cL,\lim_{t\to 0}\lambda(t).x),\lambda}$$ is linear on the cones $\overline{C_i}$, because on each cone it is given by the pairing with the weight $\chi_i$ of $\cL$ at the fixed point $x_i$. To show convexity, we use that $\cL$ is ample on $\bG_m$-orbit closures and therefore we know from \Cref{ex:weight-degree} that for all $\lambda$ and all $x^\prime \in X$ the difference
	$$\lr{\wt(\cL,\lim_{t\to 0}\lambda(t).x^\prime),\lambda} - \lr{\wt(\cL,\lim_{t\to \infty}\lambda(t).x^\prime),\lambda}>0$$
	is positive. Thus if $\lambda_0\in \overline{C}_i\cap\overline{C}_j$ lies in a face, $x^\prime=\lim_{t\to 0}\lambda_0(t).x$ is the corresponding limit point and $\lambda\in X_*(T)$ is a cocharacter that defines a line through $\lambda_0$ connecting $\overline{C}_i\cap\overline{C}_j$, the above inequality implies the convexity of our function. This shows (2) and Part (3) follows from the convexity of the function defined by the $\chi_i$.
\end{proof}
\begin{proof}[Proof of \Cref{prop:geometric_cells} (general case)]
	We have shown (1) in \Cref{lem:orbitproj} (3).  Moreover, for the proper toric variety $\overline{T.x}$ given by any orbit closure, the convexity assertion of \Cref{lem:orbitproj} shows that $\cL$ is ample on $\overline{T.x}$  (\cite[Section 3.4, Proposition]{FultonToric}). Therefore, the pull back of $\cL$ to the relative orbit closure $\overline{T.x_R}$ is ample and thus statement (2) also follows from the case of ample line bundles. 
\end{proof}
\begin{rem}\label{rem:projective_orbits}
	In the above argument we have seen that the existence of a line bundle $\cL$ on $X$ that is ample on $\bG_m$-orbit closures of closed points implies that the restriction of $\cL$ to any orbit closure of a closed point $\overline{T.x}\subseteq X$ is ample. In particular, in this case orbit closures are projective.
\end{rem}
By \Cref{prop:geometric_cells} the geometric realization $|\cell(x)|$ of a cell is a convex polytope in $X^*(T)_\bR$ and its boundary consists of the geometric realizations of the boundary of $\cell(x)$. Thus $\cC(X)$ defines a cell complex. A {\em subdivision of a cell} $c\in \cC(X)$ is a finite collection of cells $\{c_i\}_{i\in I}$ such that their geometric realizations define a subdivision of the convex polytope $|c|\subseteq X^*(T)_\bR$, i.e. such that 
$$\mathring{|c|}= \cupdot_{i\in I} \mathring{|c_i|}.$$

With these notions in place, we can now recall the definition of moment measures.
\begin{defn}[Moment measure \cite{BBQuotientsSurvey},\cite{BBS-conjecture}]\label{def:MomentMeasure}
A {\em moment measure} on $\cC(X)$ is a map 
$$m \colon \cC(X) \to \{0,1\}$$
that is not identically $0$, such that if $\{c_i\}_{i\in I}\in \cC(X)$ are cells whose geometric realizations define a subdivision of some cell $c$ and $m(c)=1$, then
$$ m (c)=  \sum_{i\in I} m(c_i).$$
A {\em geometric moment measure} is a moment measure that only takes the value $1$ on cells of maximal dimension.
\end{defn}
To a moment measure $m \colon \cC(X) \to \{0,1\}$ Bia\l ynicki-Birula and Sommese attach the open $T$-{invariant} subscheme 
$$U_{(m)}:= \{ x \in X \mid m( \cell(x^\prime)) \neq 0 \text{ for some }x^\prime \in \overline{T.x}\}\subseteq X.$$
As a moment measure is by definition not identically zero, the subset $U_{(m)}$ is non-empty. As the set of points whose cell is a fixed subset of $\pi_0(X^T)$ is constructible and $U_{(m)}$ is stable under generalizations by \Cref{prop:geometric_cells} (2), $U_{(m)}$ is a non-empty open subset of $X$. Also, by definition different moment measures define different open subsets of $X$.
\section{Proper quotients defined by moment measures}
As the definition of moment measures fits nicely with the existence criteria for good moduli spaces, we include an argument for the existence of proper moduli spaces for open substacks defined by moment measures. 

\begin{thm}[Bia\l ynicki-Birula--Sommese]
Let $X$ be a proper normal variety equipped with a torus action $T\times X \to X$ and $\cL$ a line bundle on $X$ that is ample on $\bG_m$-orbit closures of closed points. Let $m$ be a moment measure on $\cC(X)$. Then the open substack $[U_{(m)}/T] \subseteq [X/T]$ admits a proper good moduli space $[U_{(m)}/T]  \to U_{(m)}\mmod T$ that is a scheme.

The good moduli space $[U_{(m)}/T]  \to U_{(m)}\mmod T$ is a coarse moduli space if and only if $m$ is a geometric moment measure.
\end{thm}
\begin{proof}
By the main theorem of \cite{ahlh}, the open substack $[U_{(m)}/T]$ admits a proper good moduli space if and only if it is $\Theta$-reductive, $S$-complete and satisfies the existence part of the valuative criterion for properness. 

We will recall the definitions of these valuative criteria below. For each of these we let $R$ be a DVR with fraction field $K$ and $\pi\in R$ a uniformizing element.

The existence criterion for properness follows from the assumption that moment measures are non-zero on some cell of a subdivison: Let $x_K\in [U_{(m)}/T](K)$ be a $K$-point. As $T$-torsors on $\Spec(K)$ are trivial, we can lift the point to a point $\widetilde{x}_K\in U_{(m)}(K)$. As $X$ is proper this lift extends to an $R$-valued point $\widetilde{x}_R\in X(R)$ and by \Cref{prop:geometric_cells} (2) the orbits in the special fiber of the orbit closure $Z_R:=\overline{T.\widetilde{x}_R}\subseteq X \times_k \Spec R$ define a subdivision of $\cell(\widetilde{x}_K)$. As $\widetilde{x}_K\in U_{(m)}(K)$ there is a cell $c\subseteq \cell(\widetilde{x}_K)$ with $m(c)=1$. The subdivision of $\cell(\widetilde{x}_K)$ also defines a subdivision of $c$ and thus the special fiber of $Z_R$ intersects $U_{(m)}$. Therefore, a translate of $\widetilde{x}_K$ extends to an $R$-point $\tilde{x}^\prime_R \in U_{(m)}(R)$ and this defines a point of $[U_{(m)}/T](R)$ that extends $x_K$.

To show $S$-completeness let us recall that $\oST_R=[\Spec R[x,y]/(xy-\pi)/\bG_m]$ is the quotient by the $\bG_m$-action on $\Spec R[x,y]/(xy-\pi)$ for which $x$ has weight $-1$ and $y$ weight $1$. The open substack $\oST_R\smallsetminus\{0\}$ is isomorphic to a union of two copies of $\Spec(R)$ glued along their generic points. The condition that $[U_{(m)}/T]$ is $S$-complete means by definition \cite[Definition 3.38]{ahlh} that any morphism $\oST_R\smallsetminus\{0\} \to [U_{(m)}/T]$ extends uniquely to $\oST_R$.

To show this let $f\colon \oST_R\smallsetminus\{0\} \to [U_{(m)}/T]$ be a morphism, i.e., $f$ is given by two $R$-points  $x_1,x_2\in [U_{(m)}/T](R)$ together with an isomorphism of their generic points $x_{1,K}\cong x_{2,K}$.

As $T$-torsors are trivial on $\Spec(R)$ the points $x_1,x_2\in [U_{(m)}/T](R)$ admit preimages $\widetilde{x}_{1},\widetilde{x}_2\in U_{(m)}(R)$. Their generic fibers $\widetilde{x}_{1,K},\widetilde{x}_{2,K}$ differ by an element $t_K\in T(K)$, because their images in $[U_{(m)}/T](K)$ are isomorphic. Thus we can lift the morphism $f$ to a morphism
$$\tilde{f}\colon \Spec R[x,y]/(xy-\pi) \smallsetminus \{0\} \to X$$
that is equivariant with respect to a cocharacter $\lambda\colon \bG_m \to T$ satisfying $\lambda(\pi)=t_K$. As $X$ is proper, the morphism $\tilde{f}$ extends to an equivariant blow-up $$\tilde{f}\colon\Bl_{I}(\Spec R[x,y]/(xy-\pi))\to X,$$
where we may assume (\cite[Lemma 2.1]{heinloth-stability}) that the exceptional fibers of the blow-up is a chain of projective lines which is mapped to the closure of $T.\widetilde{x}_{1,K}=T.\widetilde{x}_{2,K}$ in $X_R$. Note that $\Bl_{I}(\Spec R[x,y]/(xy-\pi)) \to \Spec R$ is a partial compactification of $\bG_m \times \Spec(R) \subseteq\Spec R[x,y]/(xy-\pi)\smallsetminus\{0\}$.

As the morphism $\tilde{f}$ factors through the orbit closure $\overline{T.\tilde{x}_1}$ in $X_R$ (which coincides with the orbit closure $\overline{T.\tilde{x}_2}$), we can as in \Cref{prop:geometric_cells} replace $X$ by this closure and thus assume that $X \subseteq \bP(H^0(X,\cL))$ is projective. We write $\tilde{x}_1=[x_\chi]$ for projective coordinates of the point $\tilde{x}_1$. In this case the $T$-orbits in the image of the special fiber of $\Bl_{I}(\Spec R[x,y]/(xy-\pi)) \to \Spec R$ arise as closures of the points $\lambda(a_K).[x_\chi]$ for some $a_K\in K^*$, because the generic fiber of the blow up is $\bG_m \times \Spec K$. As in \Cref{prop:geometric_cells} this means that the support of the specialization of $\lambda(a_K).[x_\chi]$ to the special fiber of $R$ consists of those characters $\chi$ for which $\lr{\chi,\lambda}\cdot |a_K| + |x_\chi|$ is minimal. The subdivision of $\cell(\widetilde{x}_{1,K})=\cell(\widetilde{x}_{2,K})$ defined by the specialization $\overline{T.\widetilde{x}_i}$ is obtained from the convex piecewise linear function defined by the the valuations $-|x_\chi|$ and by our assumption the cells of the orbits corresponding to the extremal parts of the chain of projective lines are given by $\widetilde{x}_{i}\in U_{(m)}(R)$. As $m$ is a moment measure it takes the value $1$ only on a single cell of the subdivision of $\cell(\widetilde{x}_{1,K})$ and as $\widetilde{x}_{i}\in U_{(m)}(R)$ this implies that the cells of the extremal parts share a common cell in their closure. By convexity this implies that the same holds for all orbits of the chain, so these have to lie in $U_{(m)}$ as well. As $U_{(m)}$ was defined by a moment measure, only one orbit of the intersection of $U_{(m)}$ and the special fiber of the orbit closure $\overline{T.\tilde{x}_1}=\overline{T.\tilde{x}_2}$ is closed in $U_{(m)}$ and all other orbits in the intersection contain this orbit in their closure. As closed orbits admit affine open $T$-{invariant} neighborhoods, all orbits in the intersection have to be contained in this affine neighborhood in particular this neighborhood cannot contain projective lines. Thus all exceptional fibers of the blow-up are mapped to the same point in $X$, i.e. $\tilde{f}$ restricts to a constant morphism on the exceptional fibers of the blow{-}up. This implies that the unique extension of $\tilde{f}$ descends to a unique extension of $f$ to $\Spec R[x,y]/(xy-\pi)$, because the local description of a blow{-}up of a regular surface in a point implies that the blow{-}up defines a push-out diagram for the contraction of the exceptional fiber. This proves $S$-completeness.

The condition that $[U_{(m)}/T]$ is $\Theta$-reductive (\cite[Definition 3.10]{ahlh})  is that if $\Theta_R= [\Spec R[x]/\bG_m]$ is the quotient by the standard $\bG_m$-action on $\bA^1_R=\Spec R[x]$ for which $x$ has weight $-1$ and $0 = [\Spec{{(}R/\pi{)}}/\bG_m] \subset \Theta_R$ is the unique closed point of $\Theta_R$, then any morphism $f\colon \Theta_R\smallsetminus\{0\} \to [U_{(m)}/T]$ extends uniquely to $\Theta_R$. 
This follows by the same argument used to prove $S$-completeness, because it is again an extension property in codimension $2$.

Thus $[U_{(m)}/T]$ admits a {proper} good moduli space $U_{(m)}\mmod T$. This algebraic space is a scheme, because every closed orbit in $U_{(m)}$ admits an open $T$-{invariant} affine neighborhood which has an affine good quotient and these define an open cover of $U_{(m)}\mmod T$.

Finally the good moduli space is a geometric quotient, if and only if the preimage of every point in $U_{(m)}\mmod T$ is a single $T$-orbit, which happens if and only if the moment measure takes the value $1$ only on cells of maximal dimension, as otherwise the orbits corresponding to cells that contain a cell with value $1$ in their closure are identified. 
\end{proof}

\section{Equivariant cycle classes of orbit closures}\label{sec:equivariantclasses}

This section contains our key technical result that cycle classes of orbit closures in the equivariant cohomology of our variety $X$ are determined by the cells of the orbits (\Cref{prop:CycleClassByCell}). To prove this result, the main examples we need to compute explicitly are the equivariant cycle classes for orbit closures of linear torus actions on affine spaces (\Cref{lem:OrbitClassInV}). 

In the following we will use cohomological cycle classes for closed substacks $[Z/T]\subseteq [X/T]$. For quite general quotient stacks Edidin and Graham constructed equivariant cycle classes \cite[page 606]{EdidinGraham} and explained why they satisfy the same properties as cycle classes for schemes. 
In our special situation of a smooth global quotient stack $[X/T]$ obtained from a torus action, their general construction is explicit and allows to deduce the properties we need directly from the case of schemes: Both Chow groups and cohomology groups do not change under pull-back to a vector bundle and in any fixed degree the groups do not change when one removes a substack of high enough codimension. For the classifying stack $B\bG_m$ this shows that $\bP^n = [\bA^{n+1}\smallsetminus \{0\}/\bG_m] \subseteq [\bA^{n+1}/\bG_m]$ computes the cohomology of $B\bG_m$ up to degree $2n$ and thus $H^*(B\bG_m,\bK)$ is a polynomial ring generated by the first Chern class of the universal line bundle. Similarly fixing an isomorphism $T\cong \bG_m^r$, for all $n\geq 0$ the quotient stack $[X \times (\bA^{n+1}\smallsetminus \{0\})^r/T]$ is a smooth proper scheme and for large enough $n$ the smooth morphism $$[X \times (\bA^{n+1}\smallsetminus \{0\})^r/T] \to [X/T]$$ induces an isomorphism on cohomology groups up to any given degree. 

For a closed substack $[Z/T] \subseteq [X/T]$ of codimension $c$ one can thus define the cycle class $$\class_{[X/T]}([Z/T]) \in H^{2c}([X/T],\bK)$$ as the cycle class of the preimage of $[Z/T]$ in $H^{2c}(X \times (\bA^{n+1}\smallsetminus \{0\})^r/T,\bK)$. This allows to deduce the properties of cycle classes from the corresponding results on schemes. 

Our first ingredient is a local computation for the cycle classes of orbit closures in affine spaces. Let us introduce some notation that is convenient to formulate the result. 

As $T=\bG_m^r$, the cohomology of the classifying stack $H^*(BT,\bK)$ is a polynomial ring with generators in degree $2$. More precisely,  any character $\chi \in X^*(T)$ defines a line bundle $\cL_\chi$ on $BT$ and the morphism $$X^*(T) \to H^2(BT,\bK)$$ given by the first Chern class of line bundles $\chi \mapsto c_1(\cL_\chi)$ induces an isomorphism
$$ \Sym^\bullet X^*(T)_\bK \cong H^{2\bullet}(BT,\bK)$$
and the cohomology groups in odd degrees vanish.

If $V$ is a finite dimensional vector space with a $T$-action we will write $\Phi(V)\subseteq X^*(T)$ for the weights that occur in $V$ with positive multiplicities and for $\chi\in \Phi(V)$ we denote by $m_\chi\in \bN$ the multiplicity of the corresponding weight space. For any point $x\in V$ we will as before write $$\supp(x)\subseteq \Phi(V)$$ for the set of weights such that $x$ has a non-zero coordinate in the corresponding weight spaces. 

If $\supp(x)$ contains a basis of $X^*(T)_\bQ$ the stabilizer $\Stab_{T}(x)=\cap_{\chi\in \supp(x)} \ker(\chi)$ is finite and we write $$s_x:= \ord(\Stab_{T}(x))$$
for the order of this group scheme, i.e., the $k$-dimension of its structure sheaf. If $\supp(x)$ does not span $X^*(T)_\bQ$, let $T_x:=\Stab_{T}(x)^\circ_{\red}$ denote the reduced  connected identity component of the stabilizer, i.e. $T_x$ is the subtorus of $T$ whose character lattice is the quotient of $X^*(T)/\Span(\supp(x))$ by its torsion subgroup and define
$$s_x:= \ord(\Stab_{T}(x)/T_x).$$
In both cases $s_x$ is the order of the torsion subgroup of $X^*(T)/\Span(\supp(x))$, as this coincides with the order of $\Stab_{T}(x)/\Stab_{T}(x)^\circ_{\red}$.
	
The following lemma computes the equivariant cycle classes of orbit closures in $H^*([V/T],\bK)$.
\begin{lem}[Orbits in representations]\label{lem:OrbitClassInV}
	Let $V$ be a $d$-dimensional representation of $T$ considered as affine space over $k$, $\Phi(V)\subseteq X^*(T)$ the set of weights of $T$ occurring in $V$ and suppose that $0\not\in \Phi(V)$. Let $x\in V$ be a closed point, $r_x:=\dim T.x$ the dimension of the orbit of $x$ and $s_x$ as before the order of the torsion subgroup of $X^*(T)/\Span(\supp(x))$.
	Then we have:
	\begin{enumerate}
		\item The cycle class $ \class_{[V/T]}([\overline{T.x}/T]) \in H^{2(d-r_x)}([V/T],\bK)$ is non-zero if and only if $0\in \overline{T.x}$.
		\item The class 		
		$$s_x\cdot \class_{[V/T]}([\overline{T.x}/T])\in H^{2(d-r_x)}([V/T],\bK)=H^{2(d-r_x)}(BT,\bK)$$ 
		depends only on the extremal rays of the cone spanned by $\supp(x)$.
		\item For any subset $I\subseteq \supp(x)$ let us denote by $x_I\in V$ the projection of $x$ on the sum of the weight spaces given by the weights in $I$. \begin{enumerate}
			\item If $I\subseteq \supp(x)$ is a set of $r_x$ weights spanning a simplicial cone, then
			$$ \class_{[V/T]}([\overline{T.x_I}/T]) = \prod_{\chi \in \Phi(V)\smallsetminus I} \chi^{m_\chi} \cdot \prod_{\chi \in I} \chi^{m_\chi-1}.$$
			\item If the weights $\supp(x)$ span a strictly convex cone $C$ in $X^*(T)_{\bQ}$ and $I,J\subseteq \supp(x)$ are subsets such that the cones spanned by them define a subdivision of the cone $C$, then
			$$ s_x\cdot \class_{[V/T]}([\overline{T.x}/T]) = s_{x_{I}}\cdot \class_{[V/T]}([\overline{T.x_{I}}/T]) + s_{x_{J}}\cdot \class_{[V/T]}([\overline{T.x_{J}}/T]).$$
		\end{enumerate} 
	\end{enumerate}
\end{lem}
\begin{proof}		
Let us denote the orbit closure of the point $x \in V$ by $Z:=\overline{T.x} \subseteq V$. 

We will compute the cycle class of $[Z/T] \subseteq [V/T]$ in $H^*([V/T],\bK)\cong H^*(BT,\bK)$ by a reduction to the following simple cases:
\begin{enumerate}
	\item[(A)] If $0$ is not in the orbit closure $Z=\overline{T.x} \subseteq V$, i.e. $0\in V\smallsetminus Z$, then the cycle class vanishes, because the composition $$H^*([V/T],\bK) \to H^*([(V\smallsetminus Z)/T],\bK) \to H^*([0/T],\bK)  = H^*(BT,\bK)$$is an isomorphism and the cycle class of $[Z/T] \subseteq [V/T]$ is contained in the kernel of the first map.
	\item[(B)] If the orbit closure $Z=\overline{T.x} = V$ is the entire affine space, then the cycle class $\class_{[V/T]}([\overline{T.x}/T])$ in $H^0([V/T],\bK) \cong H^0([\overline{T.x}/T],\bK) \cong \bK$ is $1$.
\end{enumerate}
Case (A) shows that the condition in part (1) of the lemma is necessary to obtain a non-zero cycle class.
  
\vspace{0.5em}
  
\noindent{\sc Step 1: Reduction to one dimensional weight spaces.} To reduce to these cases, let $V^\prime \subseteq V$ be the $T$-{invariant} subspace of $V$ spanned by the non-zero weight components $x_\chi$ of $x$. By the Gysin sequence for the smooth pair $V^\prime \subseteq V$ we have 
$$ \class_{[V/T]}([Z/T]) = \prod_{\chi \in \Phi(V)\smallsetminus\supp(x)} \chi^{m_\chi} \cdot \prod_{\chi \in \supp(x)} \chi^{m_\chi-1} \cdot \class_{[V^\prime/T]}([Z/T]) \in H^*(BT,\bK).$$
Thus we are reduced to computing the cycle class in the case that $\Phi(V^\prime)=\supp(x)$ and that all weight spaces are $1$-dimensional. Combining the result with (B) we obtain the formula for the case of a simplicial cone claimed in (3)(a).

\vspace{0.5em}

\noindent{\sc Step 2: Reduction to finite stabilizers.} We can now reduce to the case that the stabilizer $\Stab_{T}(x)=\cap_{\chi\in \Phi(V^\prime)} \ker(\chi)$ is finite. To do this let $T_x:=\Stab_{T}(x)^\circ_{\red}$ again denote the reduced connected identity component of the stabilizer of $x$. Then $T_x$ is again a torus, so that we have a splitting $T\cong T_x\times T/T_x$ and $$[V^\prime/T]\cong BT_x \times [V^\prime/(T/T_x)].$$ Thus the class $\class_{[V^\prime/T]}([Z/T])\in H^*([V^\prime/T],\bK)=H^*(BT,\bK)$ is the image of $  \class_{[V^\prime/(T/T_x)]}([Z/(T/T_x)])\in H^*([V^\prime/(T/T_x)],\bK)= H^*(B(T/T_x),\bK)$ under the injective restriction map $H^*(B(T/T_x),\bK)\to H^*(BT,\bK)$ induced from $X^*(T/T_x)\hookrightarrow X^*(T)$. And the stabilizer of $x$ in $T/T_x$ is finite by construction. We can thus replace $T$ by $T/T_x$ to assume that the stabilizer of $x$ is finite.

\vspace{0.5em}

\noindent{\sc Step 3: Class depends on extremal rays only.} Let us show next that we can assume that all elements of $\supp(x)$ are extremal rays. Let $E\subseteq \supp(x)$ be the subset of characters that lie in extremal rays of the cone given by the support of $x$ and denote by $x_E$ the projection of $x$ on the sum of the weight spaces given by $E$. Let $s_E$ be the order of the stabilizer $\Stab_{T}(x_E)$. We claim that
$$s_E \cdot \class_{[V^\prime/T]}([\overline{T.x_E}/T]) = s_x\cdot \class_{[V^\prime/T]}([\overline{T.x}/T]).$$
To show this, consider the family of points $x_{\bA^1}\subseteq V^\prime \times \bA^1$ defined as 
$$x_{\bA^1} := \big((x_\chi, a\cdot x_{\tau})_{\chi \in E, \tau \in \Phi(V^\prime)\smallsetminus E}, a\big) \subseteq V^\prime \times \bA^1$$
and the closure $Z_{\bA^1}:= \overline{T.x_{\bA^1}} \subseteq V^\prime \times \bA^1$. The fiber over $1$ of the closure $Z_{\bA^1}$ is $\overline{T.x}=Z$.
The fiber over $0$ of the closure $Z_{\bA^1}$ contains only points of $\overline{T.x_E}$, because for any point $y_0 \in Z_{\bA^1}\cap (V^\prime \times \{0\})$ there exist a DVR $R$ over $\bA^1$ with fraction field $K$ and $t_K\in T(K)$ such that 
$$t_K.x_{K}= \big((\chi(t_K)x_\chi, a\cdot \tau(t_K)x_{\tau})_{\chi \in E, \tau \in \Phi(V^\prime)\smallsetminus E}, a\big)$$  extends to an $R$-valued point with closed point mapping to $y_0$.  
As $x_{\bA^1}$ has integral coordinates, this implies that $\chi(t_K)\in R$ for all $\chi\in E$. As the characters not contained in $E$ are convex combinations of the characters contained in $E$, $\tau(t_K)\in R$ for all $\tau\in \Phi(V^\prime) \setminus E$ as well and as the coordinate function $a$ of $\bA^1$  maps to the maximal ideal in $R$, this shows that the coordinates of any limit point have to vanish outside $E$ and so $y_0\in \overline{T.x_E}$. 

To compare the multiplicities of the fibers over $1$ and $0$, we note that $s_x$ and $s_E$ are the degrees of the restriction of the map $x_{\bA_1}$
$$\xymatrix{\bA^1 \ar[r]^-{x_{\bA^1}}\ar[dr]_-{\id} & [Z_{\bA^1}/T] \subseteq [V^\prime/T]\times \bA^1\ar[d]\\ & \bA^1}$$
to the points $1$ and $0$ over their images in $[V'/T] \times \bA^1$. Thus we find 
$$s_E \cdot \class_{[V^\prime/T]}([\overline{T.x_E}/T]) = s_x\cdot \class_{[V^\prime/T]}([\overline{T.x}/T]).$$
The same argument applies if we find colinear elements in the extremal rays $E$ of the support of $x$, as these can similarly be specialized to $x_E$ with a single non-trivial weight in the corresponding weights, by multiplying all other coordinates in the ray by $a$ as above.

This shows that the class $ s_x \cdot \class_{[V^\prime/T]}([Z/T])$ only depends on the extremal rays, which is the claim of (2). It also shows that to prove (3)(b) we may assume that $x=x_E\in V^\prime$ has only non-trivial components in extremal rays and only one non-zero component on each of these rays. 

\vspace{0.5em}

\noindent{\sc Step 4: The inductive formula.}
The strategy used above also allows us to compute the cycle class of $x_E$ by induction, using the result for simplicial cones as a starting point. 

If the cone spanned by $\supp(x_E)$ is not simplicial, let $\chi_1\in \supp(x_E)$ be some extremal ray and $\chi_2,\dots,\chi_s$ be linearly independent rays such that there exists a cocharacter $\lambda\in X_*(T)$ with 
\begin{align*}
\lr{\chi_1,\lambda}&<0 \\
\lr{\chi_j,\lambda}&=0 \text{ for }j=2,\dots,s\\
\lr{\tau,\lambda}&>0 \text{ for }\tau \in \supp(x_E)\smallsetminus\{\chi_1,\dots,\chi_s\}.
\end{align*}
Consider the family of points $x_{\bA^1}\subseteq V^\prime \times \bA^1$ defined as 
$$x_{\bA^1} := \big((x_{\chi_i}, \lambda(a).x_{\tau})_{i=1,\dots,s,\tau \in \supp(x_E)\smallsetminus\{\chi_1,\dots,\chi_s\}}, a\big) \subseteq V^\prime \times \bA^1,$$
specializing to $$x_{\lambda\leq 0}:=\big((x_{\chi_i})_{i=1,\dots,s},0\big) \in V^\prime.$$
The closure $Z_{\bA^1}:= \overline{T.x_{\bA^1}} \subseteq V^\prime \times \bA^1$ contains the second section
$$x^\prime_{\bA^1}:= \lambda^{-1}(a).x_{\bA^1}=   \big((a^{-\lr{\lambda,\chi_1}}x_{\chi_1},x_{\chi_i},x_{\tau})_{i=2,\dots,s,\tau \in \supp(x_E)\smallsetminus\{\chi_1,\dots,\chi_s\}}, a\big)$$
that specializes to $$x_{\lambda\geq 0}:=\big((0, (x_{\chi_i})_{i=2,\dots,s},x_\tau\big)_{\tau \in \supp(x_E)\smallsetminus\{\chi_1,\dots,\chi_s\}} \in V^\prime.$$
As before, the fiber over $0$ of $Z_{\bA^1}$ is the union of the orbit closures $\overline{T.x_{\lambda\leq 0}}\cup \overline{T.x_{\lambda\geq 0}}$, because the condition that for the fraction field $K$ of a DVR $R$ over $\bA^1$ and a point $t_K\in T(K)$ the element $t_K.x_{K}$ extends to an $R$-valued point and has non-zero reduction over $0$ for a basis of characters implies that the specialization is one of the two given points. 

The multiplicities are determined as before, so we find
$$ s_{x_E} \cdot \class_{[V^\prime/T]}([\overline{T.x_E}/T]))=s_{x_{\lambda \leq 0}} \cdot \class_{[V^\prime/T]}([\overline{T.x_{\lambda \leq 0}}/T]))+ s_{x_{\lambda \geq 0}} \cdot \class_{[V^\prime/T]}([\overline{T.x_{\lambda \geq 0}}/T])). $$
Thus the class is the sum of the classes obtained from any simplicial decomposition of the cone. This shows (3)(b). 

Finally, we claim that the sum of the cycles classes obtained from a decomposition of the cone defined by $\supp(x_E)=\Phi(V^\prime)\subseteq X^*(T)$ into simplicial cones is non-zero. This holds because the characters in $\supp(x_E)$ lie in an open half-space in $X^*(T)_\bQ$, as we assumed that $0\in \overline{T.x}$. Thus the monomials in (3)(a) are products of Chern classes of characters lying in an open half-space and as $\class_{[V^\prime/T]}([\overline{T.x_E}/T])\in H^*([V^\prime/T],\bK)=\Sym^\bullet X^*(T)_\bK$ is a positive linear combination of these classes $ s_E\cdot \class_{[V^\prime/T]}([\overline{T.x_E}/T]) = s_x \cdot \class_{[V^\prime/T]}([\overline{T.x}/T])\neq 0$. This finishes the proof of (1). 
\end{proof}
By \Cref{lem:OrbitClassInV} (2) we may define the cycle class of a cone of weights as the class of any point $x\in V$ such that the weights in $\supp(x)$ span the corresponding cone:
\begin{defn}[Cycle classes of cones]
		Let $V$ be a finite dimensional representation of $T=\bG_m^r$,  let $\Phi(V)\subseteq X^*(T)$ be the set of weights of $T$ occurring in $V$.
		If $\overline{C}\subseteq X^*(T)_\bR$ is a strictly convex cone spanned by extremal weights $I\subseteq \Phi(V)$  and $x_{\overline{C}}\in V$ is any point such that $\supp(x_{\overline{C}})=I$ we define the {\em cycle class of the cone} $\overline{C}$ to be:
		$$\class_{[V/T]}(\overline{C}) := s_{x_{\overline{C}}}\cdot \class_{[V/T]}([\overline{T.x_{\overline{C}}}/T]) \in H^*([V/T],\bK),$$
		where as before $s_{x_{\overline{C}}}$ is the {order of the torsion subgroup of $X^*(T)/\Span(I)$}.
\end{defn}

From \Cref{lem:OrbitClassInV} we can deduce the following result on cycle classes of cones.
\begin{lem}\label{lem:AdditiveCones}
			Let $V$ be a finite dimensional representation of $T=\bG_m^r$,  let $\Phi(V)\subseteq X^*(T)$ be the set of weights of $T$ occurring in $V$. Then the cycle classes of cones satisfy the following properties.
			\begin{enumerate}
				\item If $\overline{C}=\cup_{i=1}^n \overline{C}_i$ is a subdivision of $\overline{C}$ into strictly convex cones of the same dimension, then we have $$\class_{[V/T]}(\overline{C})=\sum_{i=1}^n \class_{[V/T]}(\overline{C}_i).$$
				\item If $X^*(T)_{\bQ}=\cup_{i=1}^n \overline{C}_i$ is a subdivision of $X^*(T)_{\bQ}$ into strictly convex cones of maximal dimension, then $$\sum_{i=1}^n \class_{[V/T]}(\overline{C}_i) =0.$$
			\end{enumerate}
\end{lem}
\begin{proof} Part (1) of the lemma follows by induction from \Cref{lem:OrbitClassInV} (3)(b), because any subdivision can be refined into a simplicial subdivision.

For (2) we show that the sum of the cycle classes is equal to the cycle class of an orbit that does not contain $0$ in its closure, by \Cref{lem:OrbitClassInV} (1) this will imply the claim. To show this choose a point $x\in V$ such that $\supp(x)\subseteq \Phi(V)$ has a single non-trivial coordinate in each extremal ray occurring in the decomposition. Then the orbit closure $\overline{T.x}$ does not contain $0$, because there does not exist a cocharacter $\lambda\in X_*(T)$ such that $\lr{\chi,\lambda}>0$ for all $\chi \in\supp(x)$. In particular the cycle class $\class_{[V/T]}([\overline{T.x}/T])=0$. 

Consider the orbit closure of the family of points given by $(ax,a)\in V\times \bA^1$. Then the fiber over $0$ is the union of the orbit closures $\overline{T.x_I}$ where $I\subseteq \supp(x)$ are those subsets, such that there exists a rational cocharacter $\lambda\in X_*(T)_{\bQ}$ such that $\lr{\chi,\lambda}=-1$ for $\chi \in I$ and $\lr{\chi,\lambda}>-1$ for $\chi\in \supp(x)\smallsetminus I$, i.e. $I$ has to be a face of the convex polytope defined by $\supp(x_I)$. As by construction the cones $\overline{C}_i$ are subdivisions of the cones defined by these faces $\sum_{i=1}^n \class_{[V/T]}(\overline{C}_i) = s_x \cdot \class_{[V/T]}([\overline{T.x}/T])=0$. 
\end{proof}
%
To translate the local computation to general $X$ we show that the cell of a point in $X$ determines the cones of weights in the Zariski tangent space to fixed points. 
\begin{prop}\label{lem:CellDeterminesRays}
	Let $X$ be a {normal} proper variety equipped with a torus action $T\times X \to X$ that admits a line bundle $\cL$ that is ample on $\bG_m$-orbit closures of closed points. Let $x\in X$ be a point, $x_i \in \overline{T.x} \cap F_i$ a fixed point of the orbit closure and $\supp_{F_i}(x)\subseteq X^*(T)$ the set of weights that appear in the tangent space $T_{\overline{T.x},x_i}\subseteq T_{X,x_i}$  to the orbit closure at $x_i$. 
	
	Then the extremal rays of the cone generated by $\supp_{F_i}(x)$ in $X^*(T)_\bR$ are those differences $\wt_{\cL}(j)-\wt_{\cL}(i)$ for $j\in \cell(x)$ that define the one dimensional faces of the polytope $|\cell(x)|$ at the vertex $\wt_{\cL}(i)$. In particular the extremal rays are determined by $\cell(x)$. 
\end{prop}
\begin{proof}
	We saw in \Cref{rem:projective_orbits} that the line bundle $\cL$ is ample on the orbit closure $\overline{T.x}$.  In the sequel we will not need that $X$ is smooth, so we replace $X$ by $\overline{T.x}$ and replacing $\cL$ by some power embed $X \subseteq \bP(H^0(X,\cL))$ equivariantly. We abbreviate $V:=H^0(X,\cL)$ and denote by $\Phi(V) \subseteq X^*(T)$ the set of weights of $T$ that appear in the $T$-vector space $V$.
	
	As the $T$-fixed points in $\bP(V)$ are given by the weight spaces, for every fixed point $x_i$ there is a unique weight $\chi_i\in \Phi(V)$ such that the point is contained in the corresponding weight space. The weights of the tangent space $T_{\bP(V),x_i}$ are thus given by the differences $\chi_j-\chi_i$ for $\chi_j\in \Phi(V)$ and $\chi_j=\chi_i$ is allowed if and only if the weight space of the character $\chi_i$ has multiplicity $\geq 2$.
	
	Let us denote by $\Phi(x)\subseteq \Phi(V)$ the weights such that the projective coordinates of $x$ have a non-zero component in the corresponding weight space. As any specialization $x\leadsto x_j$ to a $T$-fixed point can be obtained by a cocharacter $\lambda \in X_*(T)$, i.e., $x_j=\lim_{t\to 0}\lambda(t).x$, and fixed points only have a non-zero coordinate in a single weight space, the fixed points in the orbit closure are given by those elements $\chi_j\in \Phi(x)$ such that there exists a cocharacter $\lambda\in X_*(T)$ such that $$\langle \chi,\lambda \rangle > \langle \chi_j,\lambda \rangle \text{ for all } \chi \in \Phi(x)\smallsetminus \{\chi_j\}.$$
	Thus the vertices of the convex hull of $\Phi(x)$ are the weights defined by $\cell(x)$.
	
	Applying this to a fixed point $x_i\in \overline{T.x}$ implies that 
	$$\supp_{F_i}(x)\subseteq \{ \tau \mid \tau = \chi-\chi_i, \chi\in \Phi(x)\}$$ 
	because these are the weights appearing in the tangent space to $x_i$ in the projective subspace defined by the non-zero coordinates of $x$.
	
	We claim the extremal rays of the right hand side appear in $\supp_{F_i}(x)$. This holds, because the extremal rays are given by characters $\tilde{\chi}=\chi-\chi_i$ such that there exists a cocharacter $\tilde{\lambda}\in X_*(T)$ for which $\langle \tilde{\chi},\tilde{\lambda}\rangle =0$ and $\lr{\chi^\prime-\chi_i,\tilde{\lambda}}>0$ on all other rays $\chi^\prime$. This implies that $\lim_{t\to 0} \tilde{\lambda}(t).x=\tilde{x}$ exists in $X$ and the limit curve $\overline{\tilde{\lambda}.x}$ then shows that the extremal ray defined by the weights of $\tilde{x}$ appears in $\supp_{F_i}(x)$.
	
	Thus the extremal rays of $\supp_{F_i}(x)$ are determined by $\cell(x)$.    
\end{proof}
Together with the computation of cycle classes in $[V/T]$ this allows us to deduce our key result.
\begin{prop}\label{prop:CycleClassByCell}
			Let $X$ be a smooth proper variety equipped with a torus action $T\times X \to X$ that admits a line bundle $\cL$ which is ample on $\bG_m$-orbit closures of closed points. Let $x\in X$ be a point and $s_x$ the order of the quotient $\Stab_{T}(x)/\Stab_T(x)^\circ_{\red}$. Then the cycle class
		$$s_x\cdot\class_{[X/T]}([\overline{T.x}/T]) \in H^*([X/T],\bK)$$
		is determined by $\cell(x)$.
\end{prop}

\begin{proof}
	By the localization theorem for equivariant cohomology (\cite{GKM-equivariant}, \cite[Theorem 1]{BrionVergne}), the restriction map $$H^*([X/T],\bK) \to \bigoplus_{i\in \pi_0(X^T)} H^*([F_i/T],\bK)$$
	is injective.
	
	If $i\not \in \cell(x)$, the cycle class $\class_{[X/T]}([\overline{T.x}/T])$  restricts to $0$ in $H^*([F_i/T],\bK)$, because in this case the restriction map factors through $H^*([X\smallsetminus \overline{T.x}/T],\bK)$. 
	
	If $i\in \cell(x)$, we can reduce the computation of the restriction of the cycle class to $H^*([F_i/T],\bK)$ to the case of vector spaces as follows. Let $\lambda\in X_*(T)$ be a cocharacter such that $\lim_{t\to 0}\lambda(t).x =x_i\in F_i$ is the unique fixed point of the orbit closure that lies in $F_i$. Moreover, for a generic choice of such $\lambda$, the fixed point components $F_i$ coincide with components of the $\lambda$-fixed points of $X$, because a generic $\lambda$ will pair non-trivially with the finitely many characters of $T$ appearing in the normal bundles of the fixed point {components}. The Bia\l ynicki-Birula stratification of $X$ with respect to the $\bG_m$-action defined by the cocharacter $\lambda$ defines an attracting subset $$F_i^+ =\{y\in X \mid \lim_{t\to 0}\lambda(t).y \in F_i\}$$ containing $x$, that is locally over $F_i$ isomorphic to the positive part of the tangent bundle $T_{F_i}^+X\subseteq T_{F_i}X$ to $F_i$ in $X$ (\cite[Theorem 4.1 (b) and (c)]{BBactions}). The stratum $F_i^+$ comes equipped with a projection map $p_i\colon F_i^+ \to F_i$ which is an affine bundle and a section $s_i\colon F_i \to F_i^+$.
	
	Moreover, there exists an open union of strata $U=\cup_{j\in J} F_j^+ \subseteq X$, such that the inclusion $\iota_F: F_i^+ \to U$ is a closed embedding (\cite[Lemma 1.3.1]{BBS-QuotientsRk1} which is formulated over the complex numbers, in \cite[Lemma 4.11]{Quot-Gm} it is checked that this holds over any field). We compute the restriction of the cycle class using the composition of maps
	$$\xymatrix{
		{T_{F_i,x_i}^+ \cong}F^+_{i}\times_{F_i}x_i\ar[d]\ar@{^(->}[r]^-{{\iota_x}}& F_i^+\ar[d]_{p_i}\ar@{^(->}[r]^{{\iota_{F}}}  & U \ar@{^(->}[r]^{j_U} & X\\
		x_i\ar@{^(->}[r]  & F_i {\ar@/_1pc/[u]_{s_i}} & &
	}$$
	where the left square is Cartesian and $F_i^+ \times_{F_i}x_i$ is isomorphic to the positive part of the tangent space to $x_i$ with respect to the $\bG_m$-action given by the cocharacter $\lambda$.
	
	As before let us abbreviate the orbit closure by $Z=\overline{T.x}$. The restriction of the class $\class_{[X/T]}([Z/T])$ to $H^*([F_i/T],\bK)$ is given by the composition of restriction maps
	$$s_i^*\iota_F^*j_U^*\class_{[X/T]}([Z/T]) \in H^*([F_i/T],\bK).$$
	As $j_U$ is an open embedding, $j_U^*\class_{[X/T]}([Z/T])=\class_{[U/T]}([(Z\cap U)/T])$.
	
	As $\iota_F$ is a closed embedding of a smooth subscheme, the composition $\iota_F^*$ with the Gysin map $\iota_{F,*}$ is given by the cup product with the top Chern class $e_{F_i}$ of the normal bundle $T_{F_i^+}U/T_{F_i^+}F_i^+$, so
	$$ \iota_F^*\class_{[U/T]}([(Z\cap U)/T])= \class_{[F_i^+/T]}([(Z\cap F_i^+)/T]) \cup e_{F_i}.$$
	As $p_i$ is the projection from an affine bundle, the pull back $s_i^*$ is an isomorphism on cohomology. 	
	
	Finally the intersection  $Z\cap F_i^+$ is contained in the fiber $F^+_i\times_{F_i}x_i$ over $x_i$ as orbit closures can intersect fixed point components only in a single point. Thus we can again compute $\class_{[F_i^+/T]}([(Z\cap F_i^+)/T])$ as the image of the cycle class $$\class_{[F_i^+\times_{F_i}x_i/T]}([(Z\cap F_i^+ \times_{F_i} x_i)/T])\in H^*([F_i^+\times_{F_i}x_i/T],\bK) \cong H^*(BT,\bK)$$
	under the Gysin morphism $$\iota_{x,*}\colon H^*([F_i^+\times_{F_i}x_i/T],\bK) \to H^{*}([F_i^+/T],\bK) \cong H^*(F_i,\bK)\tensor H^*(BT,\bK)$$
	that raises the degree by $2\dim F_i$.
	
	Thus we find that \begin{align*}
s_i^*\class_{[F_i^+/T]}([(Z\cap F_i^+)/T]) &= (\class_{F_i}(x_i) \tensor \class_{[F_i^+\times_{F_i}x_i/T]}([(Z\cap F_i^+ \times_{F_i} x_i)/T]))\cup {s_i^*e_{F_i}}\\&\in H^*(F_i,\bK) \tensor H^*(BT,\bK).
	\end{align*}
	as the first term $\class_{F_i}(x_i)$ is a generator for the top cohomology $H^*(F_i,\bK)$ the result of the cup product $\cup {s_i^*e_{F_i}}$ only depends on the restriction of $e_{F_i}$ to $H^*(BT)$, which is given by the product of  the negative weights in $T_{F_i}X$. 
	We described the cycle class of the orbit closure of $x$ in $T^+_{F_i,x_i}$ above in \Cref{lem:OrbitClassInV}. Multiplying this class with the $e_{F_i}$ we thus obtain the equivariant cycle class of the orbit closure in $T_{F_i,x_i}\supset T_{F_i,x_i}^+$. We showed in \Cref{lem:CellDeterminesRays} that the resulting class only depends on $\cell(x)$. This shows that the restriction of the cycles class of $[Z/T]$ to $H^*(F_i,\bK)$ is also determined by $\cell(x)$.
\end{proof}
\begin{notn}[Cycle class of a cell]
	If $c\in \cC(X)$ is a cell we define $$\class_{[X/T]}(c) := s_x \cdot \class_{[X/T]}([\overline{T.x}/T]) \in H^*([X/T],\bK)$$
	for any $x\in X$ with $\cell(x)=c$. By the above proposition, this is independent of the choice of $x$.
\end{notn}
As for cones, the classes of cells are additive with respect to subdivisions.
\begin{lem}\label{lem:AddCycleClass}
	Let $X$ be a smooth proper variety equipped with a torus action $T\times X \to X$ that admits a line bundle $\cL$ which is ample on $\bG_m$-orbit closures of closed points. Let $c\in  \cC(X)$ be a cell, $\mathring{|c|}= \cupdot_{j\in I} \mathring{|c_j|}$ a subdivision of $c$ and let $c_1,\dots,c_s$ be the cells of maximal dimension in this decomposition. Then
	$$\class_{[X/T]}(c)= \sum_{j=1}^s \class_{[X/T]}(c_j).$$
\end{lem}
\begin{proof}
	As the restriction map $$H^*([X/T],\bK) \to \bigoplus_{i\in \pi_0(X^T)} H^*([F_i/T],\bK)$$
	is injective (\cite{GKM-equivariant}, \cite[Theorem 1]{BrionVergne}), it suffices to prove the equality for the restrictions of the classes to the fixed point components $F_i$.  
	
	We saw in the proof of \Cref{prop:CycleClassByCell}, that the restriction of $\class_{[X/T]}(c_j)$ to $H^*([F_i/T],\bK)=H^*(F_i,\bK)\otimes H^*(BT,\bK)$ for $i\in c_j$ can be computed as the product of the cycle class of a point in $F_i$ and the class of the cone $\overline{C}_i$ spanned by the differences $\wt_{\cL}(\ell)-\wt_{\cL}(i)$ for $\ell\in c_j$, considered as weights of the representation of $T$ given by the fiber $N_{i,x}$ of the normal bundle $N_i$ to $F_i$ at a point $x\in F_i$.
	
	As $\mathring{|c|}= \cupdot_{j\in I} \mathring{|c_j|}$ is a subdivision of the cell $c$, the restriction of $\sum_{j=1}^s \class_{[X/T]}(c_j)$ therefore vanishes in $H^*([F_i/T],\bK)$ for all $i$ that are not vertices of $c$ and coincides with the restriction of $\class_{[X/T]}(c)$ at the vertices by \Cref{lem:AdditiveCones}.
\end{proof}

\section{Proof of the main result for geometric quotients}\label{sec:geometric_quotients}

We can now prove our main result (\Cref{thm:GeometricQuotients}) that any non{-}empty open substack of $[X/T]$ that admits a proper geometric quotient is defined by a geometric moment measure.

Suppose that $j_U\colon [U/T]\hookrightarrow [X/T]$ is an open substack that admits a proper geometric quotient $p\colon [U/T] \to U/\!/T$ of dimension $\dim U/\!/T = d= \dim X-\dim T$. 

We claim that the map 
\begin{align*}
	m \colon \cC(X) &\to \{0,1\}\\
	c &\mapsto \left\{ \begin{array}{cl}
		1 & \text{if $c$ is of maximal dimension and } j_U^*\class(c) \neq 0 \in H^*([U/T],\bK)\\
		0 & \text{otherwise} 
	\end{array}\right.
\end{align*} 
is a geometric moment measure for which $U=U_{(m)}$.

To show this note that as $U/\!/T$ is irreducible and proper of dimension $d$ the top cohomology group $$H^{2d}( U/\!/T,\bK)=\bK$$ is one-dimensional and it is generated by the cycle class of any smooth closed point. Moreover, as $U\mmod T$ is a geometric quotient the map $p$ induces an isomorphism on cohomology groups
$$p^*\colon H^*( U/\!/T,\bK) \map{\cong} H^*([U/T],\bK).$$

For a point $x\in X$ the condition $x\in U$ is equivalent to the condition $T.x \subseteq U$. As $U\subseteq X$ is open and admits a geometric quotient this is equivalent to the condition $\overline{T.x} \cap U\neq 0$ and it implies that $\Stab_T(x)$ is finite of order $s_x$. Thus the cycle class $s_x \cdot \class_{[X/T]}([\overline{T.x}/T]) \in H^{2d}([X/T],\bK)$ restricts to the cycle class  $s_x \cdot \class_{[U/T]}([\overline{T.x}\cap U/T],\bK) \in H^{2d}([U/T],\bK) \cong H^{2d}(U/\!/T,\bK)$ and this class is the cycle class of a smooth closed point if and only if $\overline{T.x}\cap U$ is non-empty.

Thus we find that $x\in U$ if and only if $j_U^*\class_{[X/T]}([\overline{T.x}/T])\neq 0$. 

In particular $m$ maps the generic cell to $1$ and for any subdivision of a top-dimensional cell $c$ into top-dimensional cells $\{c_i\}_{i=1,\dots,s}$, we have $$\sum_{i=1}^s j_U^*\class_{[X/T]}(c_i)=j_U^* \sum_{i=1}^s \class_{[X/T]}(c_i) =j_U^*\class_{[X/T]}(c),$$ 
because we can check this for the restrictions of the classes to the cohomology of the fixed point components $H^*([F_i/T],\bK)$, where the result follows from the local computation for cones (\Cref{lem:AdditiveCones}). 

As for a cell of maximal dimension the restriction of $\class_{[X/T]}(c_i)$ to $H^*([U/T],\bK)$ is either $0$ or the cycle class of a smooth closed point, the function $m$ will be non-zero on exactly one of the $c_i's$. Thus $m$ is a geometric moment measure and $x\in U$ if and only if $m(\cell(x))=1$, i.e., $U=U_{(m)}$.

\section{Proof of the main result for good quotients}

Let us now show the main result \Cref{thm:GeometricQuotients} for good, but not necessarily geometric quotients.
Suppose that $j_U\colon [U/T]\hookrightarrow [X/T]$ is an open substack that admits a proper good quotient $p\colon [U/T] \to U/\!/T$. To show that the subscheme $U\subseteq X$ is defined by a moment measure, let us first show, that the points in $U$ can again be detected by a cohomological criterion. 
\begin{prop}\label{prop:nontrivclass}
	A closed point $x\in X$ is contained in $U$ if and only if the equivariant cycle class $\class_{[X/T]}([\overline{T.x}/T])\in H^*([X/T],\bK)$ of its orbit closure restricts to a non-trivial class $j_U^*\class_{[X/T]}([\overline{T.x}/T]) \neq 0 \in H^*([U/T],\bK)$.
\end{prop}
\begin{proof}
As $j_U^*\class_{[X/T]}([\overline{T.x}/T])=\class_{[U/T]}([(\overline{T.x} \cap U)/T])=0$ if $x\not\in U$ we only need to show that the cycle class does not restrict to $0$ if $x\in U$.

\vspace{0.5em}

\noindent{\sc Step 1: Points with closed orbits of maximal dimension.}
Let us denote by $$\overline{j}_{U^{st}}\colon (U\mmod T)^{st} \hookrightarrow U\mmod T$$ the (possibly empty) open subset over which the fibers of $q\colon U \to U\mmod T$ consist of a single orbit of maximal dimension and let $$j_{U^{st}}\colon U^{st}:=q^{-1}((U\mmod T)^{st})\hookrightarrow U$$ be the inclusion of the preimage in $U$:
\[
\begin{tikzcd}
	U^{st} \ar[r,hook] \ar[d,"q^{st}:=q|_{U^{st}}"'] \arrow[dr, phantom, "\ulcorner"] & U \ar[d,"q"] \\
	(U\mmod T)^{st} \ar[r,hook] & U\mmod T.
\end{tikzcd}
\]
As $q: U \to  U\mmod T$ is a good quotient, the restriction $q^{st}\colon U^{st} \to (U\mmod T)^{st}$ is a geometric quotient, i.e., $p^{st}\colon [U^{st}/T] \to (U\mmod T)^{st}$ is a coarse moduli space. 

Let us first assume that $x\in U$ is a point, such that its $T$-orbit $T.x\subseteq U$ is closed in $U$ 
and its stabilizer $\Stab_{T}(x)$ is finite. Then $x\in U^{st}$, because all fibers of $q$ contain a unique closed orbit and points with finite stabilizers are the points whose orbits have maximal dimension. In this case we can argue as in the proof for geometric quotients (\Cref{sec:geometric_quotients}) that the cycle class is a multiple of the cycle class of a smooth closed point that generates the top cohomology of $U\mmod T$: The map $p^{st}$ induces an isomorphism on cohomology sheaves $\bR p^{st}_*\un{\bK} =\un{\bK}$ and thus the induced map $p^{st,*}\colon H^*_c((U\mmod T)^{st},\bK) \to H^*_c([U^{st}/T],\bK)$ is an isomorphism. As $U\mmod T$ is irreducible and proper, the morphism
$$H^*_c([U^{st}/T],\bK)\cong H^*_c((U\mmod T)^{st},\bK) \to H^*(U\mmod T,\bK)$$ 
is an isomorphism in the top cohomological degree $2\dim U\mmod T$ and this group is generated by the cycle class of any smooth closed point. As the class $s_x\cdot \class_{[U^{st}/T]}([T.x/T])$ is independent of the point $x\in U^{st}$ by \Cref{prop:CycleClassByCell}, it is mapped to a non-zero class in $H^*(U\mmod T,\bK)$ under the above morphism. The image of this non-zero class under the pull back map $p^*\colon H^*(U\mmod T,\bK)\to H^*([U/T],\bK)$ coincides with the cycle class $s_x\cdot \class_{[U/T]}([T.x/T])$, because for a point $x\in U^{st}$ mapping to a smooth point $q(x)$ in $(U\mmod T)^{st}$ we get a Cartesian diagram:
$$\xymatrix{
	[T.x/T] \ar[d]\ar@{^(->}[r]^{\iota_x} & [U/T] \ar[d]^p \\
	 q(x) \ar@{^(->}[r]^{\iota_{q(x)}} & U\mmod T
}$$
and the corresponding Gysin morphisms $\iota_{x,!}\iota_{x}^! \un{\bK} \to \un{\bK}$ and $\iota_{q(x),!}\iota^{!}_{q(x)} \un \bK \to \un{\bK}$ from which the cycle classes are defined, can be computed locally. As this construction commutes with smooth base change, and $p$ is smooth around $q(x)$ we conclude that $s_x\cdot \class_{[U/T]}([T.x/T])$ is the image of a generator of the top cohomology of $U\mmod T$ under the pullback map.

Finally the top cohomology of $H^*(U\mmod T,\bK)$ is a direct summand of $H^*([U/T],\bK)$, because the one dimensional top cohomology of $H^*(U\mmod T,\bK)$ coincides with intersection cohomology and this is a direct summand of $H^*([U/T],\bK)$ by \cite[Corollary 3.5]{Woolf} in characteristic $0$. The argument in \cite{Kinjo-decomposition_good_moduli} comparing the cohomology sheaves of the quotient map $[U/T]\to U\mmod T$ with the cohomology of a projective morphism as indicated in  \Cref{sec:equivariantclasses} also applies for étale cohomology, see \Cref{Prop:semisimplepure}. This shows that the cycle class $\class_{[U/T]}([T.x/T])$ in $H^*([U/T],\bK)$ is also non-zero. 

\vspace{0.5em}

\noindent{\sc Step 2: Points with closed orbits.}
If the $T$-orbit of $x$ in $U$ is closed, but $\Stab_{T}(x)$ is not finite, let us again denote by $T_x:=\Stab_{T}(x)^\circ_\text{red} \subseteq T$ the reduced identity component of the stabilizer, which is a subtorus of $T$. As before the fixed point set $X^{T_x}\subseteq X$ is a union of smooth connected subschemes of $X$ and as $T$ is abelian and connected, both $X^{T_x}$ and its components are $T$-{invariant}. Moreover $T_x$ acts with non-trivial characters on the normal bundle of $X^{T_x}$ in $X$. Let $F \subseteq X^{T_x}$ be the irreducible component containing $x$. As $F \subseteq X$ is closed, so is $[(F \cap U)/T] \subseteq [U/T]$ and therefore the (scheme theoretic) image of {$[(F\cap U)/T]$} in $U\mmod T$ is a good moduli space (\cite[Lemma 4.14]{alper-good}) and it is closed in $U\mmod T$ (\cite[Theorem 4.16]{alper-good}), so it is again proper. We denote this good moduli space by $(F \cap U)\mmod T$.

As $T_x$ acts trivially on $F$ the action of $T$ on $F$ factors through the quotient $T/T_x$ and $(F \cap U)\mmod T \cong (F \cap U)\mmod(T/T_x)$ is a good quotient for the $T/T_x$-action. The action also defines a Cartesian diagram
$$\xymatrix{
[F/T]\ar[r]\ar[d] & BT \ar[d]\\ [F/(T/T_x)]\ar[r] & B(T/T_x)}$$
and as the cohomology of $BT$ is the symmetric algebra on $X^*(T)_\bK$, base change for the proper horizontal maps shows that $$H^*([F/T],\bK) \cong H^*( [F/(T/T_x)],\bK) \tensor H^*(BT_x,\bK).$$

As the stabilizer of $x$ in $T/T_x$ is finite, the orbit $(T/T_x).x=T.x \subseteq F \cap U$ is closed and $(F \cap U)\mmod (T/T_x)$ is proper, we can apply the first step of the proof to obtain that {$\class_{[(F\cap U)/(T/T_x)]}([T.x/(T/T_x)])\in H^*([(F\cap U)/(T/T_x)],\bK)$} is non-trivial.

As $[(F\cap U)/T] \cong [(F\cap U)/(T/T_x)] \times BT_x$, the pull back map $$H^*([(F\cap U)/(T/T_x)],\bK) \to H^*([(F\cap U)/T],\bK)$$ is injective and commutes with the formation of cycle classes. Thus the result for $\class_{[(F\cap U)/(T/T_x)]}([T.x/(T/T_x)])$ implies the same result for $\class_{[(F\cap U)/T]}([T.x/T])\in H^*([(F\cap U)/T],\bK)$.

The cycle class $\class_{[U/T]}([T.x/T])\in H^{*}([U/T],\bK)$ is the image of this class under the Gysin map $H^{*-2c}([(F\cap U)/T],\bK) \to H^{*}([U/T],\bK),$ where $c=\dim U-\dim F \cap U$. This Gysin map is injective by the argument for Kirwan surjectivity (\cite{kirwan-surj}), because the composition 
\begin{align*}
H^{*-2c}([(F\cap U)/T],\bK) \to H^{*}([U/T],\bK) &\to H^{*}([(F\cap U)/T],\bK)\\&\cong H^{*}([(F\cap U)/(T/T_x)],\bK)\tensor H^*(BT_x,\bK)
\end{align*} of the Gysin map with the restriction map is the cup product with the top Chern class of the normal bundle of $F \cap U \subseteq U$. As $T_x$ acts with non-trivial characters on the normal bundle and the Künneth-component of the top Chern class in $H^0([(F\cap U)/(T/T_x)],\bK)\tensor H^*(BT_x,\bK)$ is given by the product of these characters, the cup product with this class is injective. 

This shows that the equivariant cycle classes of closed orbits in $U$ are non-trivial in $H^*([U/T],\bK)$ if $U\mmod T$ is proper.

\vspace{0.5em}

\noindent{\sc Step 3: Points with non-closed orbits.} If the orbit of $x$ is not closed in $U$, the orbit closure $\overline{T.x}\cap U$ in $U$ contains a unique closed orbit $T.x_1$, because $[U/T]\tto U\mmod T$ is a good quotient (\cite[Theorem 4.16]{alper-good}). As the dimension of the closed orbit $T.x_1$ is smaller than the dimension of $T.x$, the stabilizer $\Stab_{T}(x)\subsetneq \Stab_{T}(x_1)$ of $x_1$ is of larger dimension than the stabilizer of $x$. In particular the stabilizer of $x_1$ has positive dimension.

To conclude that the cycle class $\class_{[U/T]}([(\overline{T.x}\cap U)/T])\in H^*([U/T],\bK)$ is non-zero we argue by localization to the fixed point stratum of the stabilizer of the closed orbit $T.x_1$.
	
Again let $T_1:=\Stab_{T}(x_1)^\circ_{\red}$ be the reduced identity component of the stabilizer, let $F \subseteq X^{T_1}$ be the irreducible component of the $T_1$ fixed point locus in $X$ that contains $x_1$ and let $F_U:=F\cap U$ denote the intersection of $F$ with $U$, so that $F_U\mmod T \subseteq U\mmod T$ is a proper good moduli space. 

We claim that the restriction of the cycle class $\class_{[X/T]}(\overline{T.x}/T)$ to $$H^*([F_U/T],\bK)\cong H^*([F_U/(T/T_1)],\bK)\tensor H^*(BT_1,\bK)$$ is of the form $\class_{[F_U/T]}([T.x_1/T]) \cup e$, where $e$ is not a zero-divisor. As we already proved that the cycle classes of closed orbits are non-zero, this will prove our claim that the cycle class of $\overline{T.x}$ is non-zero as it restricts to a non-zero class in $H^*([F_U/T],\bK)$.

As $T.x$ specializes to the closed orbit $T.x_1 \subseteq U$ we may choose the point $x_1$ defining the closed orbit so that, there exists a cocharacter $\lambda \colon \bG_m \to T$ satisfying $x_1 = \lim_{t\to 0}\lambda(t).x$. As $x_1$ is fixed by $\lambda$, we have that $\lambda\colon \bG_m \to T_1 \subseteq T$ is a cocharacter for $T_1$.

Let us denote by $F^+_U \subseteq U$ the locus of points $y\in U$ such that $  \lim_{t\to 0}\lambda(t).y \in F_U$, i.e., the Bia\l ynicki-Birula stratum corresponding to $F_U$, which comes equipped with a retraction $\pi\colon F_U^+\to F_U$. In particular, locally on $F_U$ we can describe  $F_U^+$ in terms of the normal bundle to $F_U$, namely let $$T_{F_U}= \bigoplus_{\chi \in X^*(T_1)} T_{F_U,\chi}$$
be the decomposition of the tangent bundle into weight spaces for the $T_1$-action. Then 
$$T_{F_U}^+ := \bigoplus_{\chi \in X^*(T_1) \atop \langle \chi,\lambda\rangle>0} T_{F_U,\chi}$$
is locally over $F_U$ equivariantly isomorphic to $F_U^+$ \cite[Theorem 4.1 (b)]{BBactions}.
 
As $T$ is abelian and $F_U$ is $T$-invariant, the stratum $F_U^+$ is $T$-invariant as well. Also $F_U^+\subseteq U$ is closed in $U$, because $U$ admits a good quotient and is therefore $\Theta$-reductive (\cite[Proposition 3.13 and Remark 3.14]{ahlh}), in terms of orbit closures we can also see this directly, because any orbit closure $\overline{T.y}$ contains a unique closed orbit in $U$, but as $F_U\subseteq F$ is closed orbit closures of points in $F_U^+$ have to lie in $F_U$. 	

If we denote $F_{T.x_1}^+:=\pi^{-1}(T.x_1)$ we thus obtain a commutative diagram:
$$\xymatrix{
 \overline{T.x}\cap U \ar@{^(->}[r]\ar[dr] & F^+_{T.x_1}\ar[d]^{\pi}\ar@{^(->}[r] & F_U^+\ar[d]^{\pi}\ar@{^(->}[r] & U\ar@{=}[d]\\
 & T.x_1\ar@{^(->}[r]  & F_U\ar@{^(->}[r]  & U	
}$$
in which the horizontal arrows are closed embeddings.

The cycle class of $[(\overline{T.x}\cap U)/T]$ in $H^*([U/T],\bK)$ can thus be computed from $\class_{[F_{T.x_1}^+/T]}([(\overline{T.x}\cap U)/T])\in H^*([F_{T.x_1}^+/T],\bK)$ as the image under the composition of the  Gysin homomorphisms 
$$ H^*([F^+_{T.x_1}/T],\bK) \to H^{*+2f}([F^+_{U}/T],\bK) \to 
H^{*+2f+2c}([U/T],\bK).$$
where $f=\dim F_U^+-\dim F_{T.x_1}^+=\dim F_U-\dim T.x_1$. As in Step 2, the second morphism is injective, because the composition of the Gysin homomorphism with the restriction to $H^*([F^+_U/T],\bK)$ is the cup product with the top Chern class of the normal bundle to $F_U^+$. Now $\pi$ is an affine fibration so $$H^*([F^+_U/T],\bK)\cong H^*([F_U/T],\bK)$$ and the restriction of the normal bundle of $F_U^+$ to $F_U$ is 
$$T_{F_U}^- := \bigoplus_{\chi \in X^*(T_1) \atop \langle \chi,\lambda\rangle<0} T_{F_U,\chi},$$
so that the stabilizer $T_1$ again acts with non-trivial characters on the normal bundle and therefore the top Chern class is not a zero divisor in $H^*([F_U/T],\bK)=H^*([F_U/(T/T_1)],\bK) \tensor H^*(BT_1,\bK)$. 

The cycle class of $[F^+_{T.x_1}/T]$ in $H^*([F_U^+/T],\bK)\cong H^*([F_U/T],\bK)$ agrees with the cycle class of $[T.x_1/T]$ in $H^*([F_U/T],\bK)$, as $F^+_{T.x_1}$ is the pull-back of $T.x_1$. 

To compute the cycle class $\class_{[F_{T.x_1}^+/T]}([(\overline{T.x}\cap U)/T])$ we note that as $[T.x_1/T]= B\Stab_{T}(x_1)$ is the classifying stack of a multiplicative group, the restriction of the fibration $\pi$ to $T.x_1$ is equivariantly isomorphic to the positive part of the tangent space $T_{F_U,x_1}^+:=V_1$, i.e., $$F_{T.x_1}^+\cong T.x_1 \times V_1$$
where $V_1$ is a representation of $T$ on which $T_1\subseteq T$ acts with non-trivial characters only. Splitting $T\cong T_1 \times T/T_1$ we see that the closure $\overline{T.x}$ in $F_{T.x_1}^+\cong T.x_1 \times V_1$ is the product of $T.x_1$ with the closure of the $T_1$-orbit in $V_1$, for which we computed the cycle class in \Cref{lem:OrbitClassInV} as non-trivial expression in $H^*(BT_1,\bK)$ in terms of the characters of $T_1$ that appear in $V_1$. As these classes are the restriction of the first Chern classes of the bundles $T_{F_U,\chi}$ which are non-zero divisors in $H^*([F_U/T],\bK)$ this proves our claim that we can express $\class_{[F_U^+/T]}([\overline{T.x}\cap U/T])$ as product of the cycle class of $[T.x_1/T]$ with a non-zero divisor.
\end{proof}

As by \Cref{prop:CycleClassByCell} cycle classes of orbit closures only depend on the cell of the corresponding point, the above proposition again shows that the condition for a point $x\in X$ to lie in $U$, only depends on $\cell(x)$. This allows us to construct a moment measure from $U$.
\begin{prop}
	Let $X$ be a smooth proper variety equipped with a torus action $T\times X \to X$ that admits a line bundle $\cL$ which is ample on $\bG_m$-orbit closures of closed points and $U\subseteq X$ an open $T$-invariant subscheme that admits a proper good quotient $[U/T]\to U\mmod T$. Then the map 
	\begin{align*}
		m \colon \cC(X) &\to \{0,1\}\\
		c &\mapsto \left\{ \begin{array}{cl}
			1 & \text{if }c=\cell(x) \text{ for some }x \in U \text{ such that }T.x \subseteq U\text{ is closed}\\
			0 & \text{otherwise} 
		\end{array}\right.
	\end{align*} 
	is a moment measure and $U=U_{(m)}$.
\end{prop}  
\begin{proof}
	We have already seen that $U_{(m)}=U$: For any $x \in U$, there exists a unique point $x' \in \overline{T.x} \cap U$ such that $T.x' \subseteq U$ is closed,  since $[U/T] \to U\mmod T$ is a good quotient. Then $m(\cell(x'))=1$ and by definition $x \in U_{(m)}$. Conversely, for any $x \in U_{(m)}$ by definition $m(\cell(x'))=1$ for some $x' \in \overline{T.x}$. As the equivariant cycle class of $x^\prime$ is determined by $\cell(x^\prime)$, \Cref{prop:nontrivclass} implies that $x' \in U$ and hence $x \in U$ since $U$ is open and $T.x^\prime$ is in the closure of $T.x$.
	
	It remains to prove that $m$ is a moment measure, i.e. that if $m(c)=1$ for some cell $c$ and $\mathring{|c|}=\cupdot \mathring{|c_i|}$ is a subdivision of $c$ then $m(c_i)=1$ for exactly one $i$.
	
	We will first prove this in the case that $U$ contains points with closed orbits that define maximal dimensional cells. 

\vspace{0.5em}
	
	\noindent{\sc Step 1: Reduction of the claim for cells of maximal dimension to the case of the generic cell.} First assume that $c=\cell(x)$ is a cell of maximal dimension satisfying $m(c)=1$, i.e., $x$ is a point with closed orbit of maximal dimension in $U$.  In this case, we can reduce to the case that $c$ is the generic cell: As any point $x$ can be obtained as specialization of the generic point and specializations define subdivisions (\Cref{prop:geometric_cells}), there is a subdivision of the generic cell $\mathring{|c_{gen}|}=\cupdot \mathring{|\cell(x_i)|}$ such that $x=x_i$ for some $i$. Thus any subdivision of $c$ induces a refinement of this subdivision of $c_{gen}$.

We claim that if $U$ contains a point $x$ with closed orbit of maximal dimension, then the generic orbit is closed in $U$ and thus $m(c_{gen})=1$, which finishes the reduction step. To see this, write $x$ as specialization of the generic point $x_{gen} \in U$, i.e. choose a DVR $R$ with fraction field $K$ and a point $x_R\in U(R)$ such that $x_K$ lies over the generic point of $U$ and the special fiber of $x_R$ maps to $x$. As $p\colon [U/T] \to U\mmod T$ is a good quotient, there is a unique orbit $T.x^\prime_K$ in the closure of the generic orbit, that is closed in $U_K$. As $p$ is universally closed, this orbit will specialize to some orbit $T.x_0$ in the special fiber of $\overline{T.x_R}\cap U_R$. If $x^\prime_K$ is not the generic orbit, then its stabilizer is of positive dimension so the same would hold for its specialization $x_0$. However, as $U\mmod T$ is separated, only one closed orbit in the special fiber of $\overline{T.x_R}\cap U_R$ can lie in $U$ and therefore the specialization $T.x_0$ is necessarily $T.x$. As the stabilizer of $x$ is finite, this implies that $T.x_K^\prime$ is the generic orbit.

\vspace{0.5em}

\noindent{\sc Step 2: Reduction to a computation of classes of stars of a cell.} If $m(c_{gen})=1$, a general orbit in $U$ is closed. Thus as in \Cref{prop:nontrivclass} the open subset $$\overline{j}_{U^{st}}\colon (U\mmod T)^{st} \hookrightarrow U\mmod T$$ over which the fibers of $p\colon U \to U\mmod T$ consist of a single orbit of maximal dimension is non-empty and $j_U^*\class_{[X/T]}(c_{gen})\in H^{2\dim([U/T])}([U/T],\bK)$ is the image of the cycle class of a smooth closed point of $U\mmod T$ under the morphism $$p^*\colon H^{2\dim(U\mmod T)}(U\mmod T,\bK) \to H^{2\dim(U\mmod T)}([U/T],\bK).$$ So we need to prove that for any subdivision $\mathring{|c_{gen}|}=\cupdot_{i\in I} \mathring{|c_i|}$ of the generic cell, there is a unique $j\in I$ such that $m(c_j)=1$, i.e. such that $c_j$ is the cell of a point with closed orbit in $U$. 

To prove this, we mimic the argument for geometric moment measures, where we used that for geometric quotients cells corresponding to closed orbits of maximal dimension in $U$ all give rise to the same cycle class. For cells $c_j$ that are not of maximal dimension we will show that the same result holds for the sum of the maximal dimensional cells appearing in the star neighborhood of $c_j$.

Suppose $m(c_j)=1$ for some $j$ then no cell in the boundary of $c_j$ corresponds to points in $U$, because $c_j$ corresponds to a closed orbit in $U$ by definition and the boundary of a cell corresponds to cells of points in the orbit closure. 

Let us denote by $\Star(c_j)$ the union of the cells appearing in the subdivision $\mathring{|c_{gen}|}=\cupdot_{i\in I} \mathring{|c_i|}$, that contain $c_j$ in their closure. Then by \Cref{prop:nontrivclass} all cells $c_{i}\in \Star(c_j)$ also satisfy $j_U^*\class_{[X/T]}(c_i)\neq 0 \in H^*([U/T],\bK)$, because $U$ is open and thus a point lies in $U$ if and only if some orbit in its orbit closure lies in $U$. 

Moreover, as $p\colon [U/T] \to U\mmod T$ is a good quotient, all orbit closures in $U$ contain a unique closed orbit. This implies that if $m(c_{j^\prime})=1$ for some other cell, then $\Star(c_j)\cap \Star(c_{j^\prime})=\varnothing$, as any point corresponding to a cell in this intersection would be in $U$ and contain at least 2 closed orbits in its closure. 

As $\class_{[X/T]}(c_{gen})$ is the sum of the classes of the maximal dimensional cells in the decomposition $\mathring{|c_{gen}|}=\cupdot_{i\in I} \mathring{|c_i|}$ by \Cref{lem:AddCycleClass}, it thus suffices to show that $$j_U^*\class_{[X/T]}(\Star(c_{j}))=j_U^*\class_{[X/T]}(c_{gen})\in H^*([U/T],\bK)$$ whenever $m(c_j)=1$, because we just saw that any two stars of classes satisfying $m(c_j)=1$ are disjoint and cells that do not contain a cell $c^\prime$ with $m(c^\prime)=1$ map to $0$ in $H^*([U/T],\bK)$.

We have already seen this if $c_j$ is a cell of maximal dimension satisfying $m(c_j)=1$, because then $\Star(c_j)=c_j$ and $c_j$ corresponds to a closed orbit that appears as specialization of the generic orbit. Thus in the cell decomposition of $c_{gen}$ defined by a specialization to $c_j$ the only maximal dimensional cell corresponding to closed orbits in $U$ is $c_j$.

\vspace{0.5em}

\noindent{\sc Step 3: Identifying the class of a star of cells.} To compare the class of a star $j_U^*\class_{[X/T]}(\Star(c_{j}))$ for orbits that are not of maximal dimension with the class of maximal dimensional orbits, we use a Kirwan blow-up to increase the dimension of the closed orbit. Suppose that $c_j=\cell(x)$ is a cell satisfying $m(c_j)=1$, which is not of maximal dimension. Then the subtorus $$T_j := \big(\left\{ t\in T \mid \chi_i(t)=\chi_{i^\prime}(t) \text{ for all }i,i^\prime \in c_j\right\}\big)^\circ_{\text{red}} \subseteq T$$
is the reduced connected identity component of $\Stab_{T}(x)$, because the dimensions of the tori agree and cocharacters in $T_j$ act trivially on $\overline{T.x}$ by our positivity assumption on the line bundle $\cL$ and the basic example \Cref{ex:weight-degree}.

Now let $X^{T_j}=\coprod F^\prime_\ell$ be the connected components of the fixed point components of $T_j$ and let $F_x^\prime$ denote the component containing $x$. As $F_x^\prime$ contains exactly those $T$-fixed point components $F_j$ for $j\in \cell(x)$, the {component} only depends on the cell $c_j$, not on the choice of $x$.

Let $b\colon \widetilde{X}:= \Bl_{F_x^\prime}(X)\to X$ be the blow up of $X$ in $F_x^\prime$. Then the preimage $b^{-1}(F_j)$ of a $T$-fixed point stratum $F_j\subseteq F_x$ is the projectivization of the normal bundle to $F_x^\prime$. Let $\chi_\ell^\prime \in X^*(T)$ for $\ell=1,\dots, f$ be the characters appearing in the fibers of the normal bundle, so that the {$T$-}fixed point components
$$ 	 b^{-1}(F_j)^T = \coprod_{\ell=1}^f F_{j,\ell}$$
of $\widetilde{X}$ above $F_j$ are the projectivizations of the subbundles of the normal bundle corresponding to the characters $\chi'_\ell$. Because in projective spaces all subsets of characters that define the vertices of a convex polytope define cells, the same holds for the cells of orbits contained in $b^{-1}(F_j)$. More precisely, the map $b$ induces a map on the cell complexes $\cC(\widetilde{X})\to \cC(X)$ and the preimage of a cell can be computed as in the description of blow ups of toric varieties.

We choose for all maximal dimensional cells $c_i$ of our subdivision of $c_{gen}$ a point $x_i\in X$ satisfying $c_i=\cell(x_i)$. Then $x_i\not\in F_{x}^\prime$ as the stabilizer of $x_i$ is finite, so $x_i$ has a unique preimage in $\widetilde{X}$ and we denote by $\tilde{c}_i\in \cC(\widetilde{X})$ the cell of this preimage. Then the union $\cup_{j\in J} \tilde{c}_i$ can be completed to a subdivision $(\cup_{j\in J} \tilde{c}_i)\cup(\cup_{\ell \in L} \tilde{d}_\ell)$ of the generic cell $\widetilde{c}_{gen}$ in $\widetilde{X}$ by adding in cells $\tilde{d}_\ell$ of $\cC(b^{-1}(F_j))$ from all $j$ such that $F_j \subseteq F_x$, because we saw that all convex subsets of fixed point components in $b^{-1}(F_j)$ appear as cells.

Let $b^{-1}(\Star(c_i))$ be the subset of cells in this decomposition of $\widetilde{c}_{gen}$ that map to $\Star(c_i)$. We claim that $\class_{[X/T]}(b^{-1}(\Star(c_i)))= b^*(\class_{[X/T]}(\Star(c_i)))\in H^*([\widetilde{X}/T],\bK)$. This holds, because the class $\class(\Star(c_i))$ only restricts to a non-trivial class in $H^*([F_j/T],\bK)$ for the vertices in the star and these are not contained in $F_x^\prime$. Similarly, $\class_{[X/T]}(b^{-1}(\Star(c_i)))$ only restricts to a non-trivial class in fixed point components outside of the exceptional divisor and there the classes agree by construction. 

Now by \cite[Proposition 3.4]{EdidinRydh} there is a canonical open substack $\Bl_{F_{x}^\prime\cap U}^p([U/T]) \subseteq [\widetilde{X}/T]$ called the saturated blow{-}up, that has a good moduli space $\tilde{p}\colon \Bl_{F_{x}^\prime\cap U}^p([U/T]) \to \Bl_{p(F_{x}^\prime\cap U)}(U\mmod T)$ which is a blow{-}up centered in $p(F_{x}^\prime)$ and in this saturated blow{-}up the dimension of the stabilizer dropped over $X^{T_j}$ (\cite[Theorem 2.11 (3)]{EdidinRydh}). As $b^*$ maps the cycle class of a general orbit in $X$ to the cycle class of a general orbit in $\widetilde{X}$ by induction this reduces the claim to the case that the $c_i$ is a cell of maximal dimension. 

This concludes the argument for subdivisions of cells $c$ that have maximal dimension and satisfy $m(c)=1$. 

\vspace{0.5em}

\noindent{\sc Step 4: Subdivisions of cells that are not of maximal dimension.}
In general, if we start with a subdivision $\mathring{|c|}=\cupdot \mathring{|c_i|}$ of a cell $c$ that is not of maximal dimension, let $T_c\subseteq T$ be the torus defined by $c$ and $F_c \subset X^{T_c}$ the component of the fixed point set of $T_c$ that contains the points with cell $c$, i.e. the component containing the $T$-fixed point components  $F_i$ for $i\in c$.

Then the cycle class $\class_{[X/T]}(c)\in H^*([X/T],\bK)$ is the image of the cycle class $\class_{[F_c/T]}(c)\in H^*([F_c/T],\bK)$ under the Gysin map. By definition, the cell $c$ corresponds to an orbit of maximal dimension in $F_c$ and any cell decomposition of $c$ appears as a cell decomposition in $\cC(F_c)$. As before $U\cap F_c \subseteq U$ is closed and $T$-invariant and thus admits a good moduli space $(U\cap F_c) \mmod T$. Thus we can apply the result for cells of maximal dimension to $U\cap F_c$ and obtain that $m(c_j)=1$ for a single cell of the subdivision of $c$.
\end{proof}
\appendix
\section{Comparison of cohomology of quotient stacks and good moduli spaces}
In this {appendix} we provide a proof of the result that the intersection cohomology of a good quotient is a direct summand of the cohomology of the stack for the case of torus quotient over fields of any characteristic. The argument has appeared recently in an article of Kinjo \cite{Kinjo-decomposition_good_moduli} for general good quotients  and independently by Hennecart \cite{Hennecart-CohomologicalIntegrality} for global quotients, both working over the complex numbers. Our argument, included upon request, is a variation of the same argument,  we only have to replace the use of mixed Hodge modules by $\ell$-adic sheaves to apply the decomposition theorem of \cite{BBDG}, Hennecart already indicates this in \cite{Hennecart-CohomologicalIntegrality}. To keep the argument short and  self-contained, we only consider the case of a global quotient by a torus action, where we can use explicit proper approximations. Throughout we will denote the constant sheaf on a scheme or  stack $\cX$ by $\un{\bK}_{\cX}$.
\begin{prop}\label{Prop:semisimplepure}
Let $U$ be a smooth variety equipped with a torus action $T\times U\to U$ that admits a good moduli space $p\colon [U/T]\tto U\mmod T$ such that for a $T$-invariant dense open subset $U^{st}\subseteq U$ the map $[U^{st}/T] \to U^{st}\mmod T$ is a coarse moduli space.

Then the direct image $\bR p_*\ubK_{[U/T]}[\dim [U/T]]$ is a semisimple pure complex that contains the intersection complex $\IC_{U\mmod T}$ of $U\mmod T$ as a direct summand.
\end{prop}
As in \cite{Kinjo-decomposition_good_moduli} and \cite{Hennecart-CohomologicalIntegrality}, the key ingredient of the argument is a geometric approximation argument similar to the one used in \Cref{sec:equivariantclasses} for Chow groups. More precisely for any integer $c>0$ we construct an action of $T$ on a projective scheme $\bP_N$ such that for the diagonal action on $U\times \bP_N$ there exists a closed subset $Z\subseteq U\times \bP_N$ of codimension at least $c$ such that the complement $U\times \bP_N\smallsetminus Z$ admits a geometric quotient $(U\times \bP_N\smallsetminus Z)\mmod T$ that is locally projective over $U\mmod T$. This allows us to compare the direct image $\bR p_* \ubK_{[U/T]}[\dim [U/T]]$ with the direct image of a proper morphism to which the decomposition theorem applies.

For any $N\in \bN$ let us denote by $V_{N,\pm}=k^N\oplus k^N$ the $2N$-dimensional representation of $\bG_m$ that is of weight $1$ on the first $N$ coordinates and of weight $-1$ on the last $N$ coordinates. Then the line bundle $\cO(1)$ on $\bP(V_{N,\pm})$ carries a standard linarization of $\bG_m$ for which the weight at the fixed points $\bP(k^N \oplus 0)\subseteq \bP(V_{N,\pm})$ is $1$ and the weight at the fixed points $\bP(0\oplus k^N)\subseteq \bP(V_{N,\pm})$ is $-1$. {As the invariant elements of the homogeneous coordinate ring of $\bP(V_{N,\pm})$ are generated by the products of pairs of coordinates with opposite weights,} the unstable points for $\bG_m$-action with respect to the standard linearization of $\cO(1)$ are exactly the fixed points and the same applies for the standard linearization of $\cO(d)$ for $d>0$.

Similarly, choosing a splitting of our torus $T\cong \bG_m^r$ we denote by $$\bP_{N}:=\prod_{i=1}^r \bP(V_{N,\pm})$$ the product of the projective spaces $\bP(V_{N,\pm})$ equipped with the $T$-action such that $T$ acts on the $i$-th $\bP(V_{N,\pm})$-factor of $\bP_{N}$ through the projection to the $i$-th $\bG_m$-factor of $T$. Then the line bundle $\cL:=\cO(1,\dots,1)$ and its powers $\cL^d$ on $\bP_N$ again come equipped with a standard linearization of $T$.

Any character $\chi\in X^*(T)$ defines a linearization $\cO_\chi$ of the trivial line bundle and we write $\cL^d_\chi$ for the corresponding change of linearization on $\cL^d$. The following is a variant of \cite[Lemma 3.1]{Kinjo-decomposition_good_moduli} and \cite[Section 5]{Hennecart-CohomologicalIntegrality}.
\begin{lem}\label{lem:ProperCover}
	Let $U$ be a normal variety equipped with a torus action $T\times U\to U$ that admits a good moduli space $p\colon [U/T]\tto U\mmod T$ such that for a $T$-invariant dense open subset $U^{st}\subseteq U$ the map $[U^{st}/T] \to U^{st}\mmod T$ is a coarse moduli space. 
	
    For any integer $c>0$ there exist $N>0$ and a linearization $\cL^d_{U,\chi}$ of some power of $\cL_{U}:=p_{\bP_N}^*\cO(1,\dots,1)$ on $U\times \bP_{N}$ such that
    \begin{enumerate}
    	\item the $\cL_{U,\chi}^d$-unstable locus $Z$ of $U\times \bP_{N}$ has codimension $\geq c$,
    	\item all $\cL_{U,\chi}^d$-semistable points of $U\times \bP_{N}$ are stable, so that $U\times \bP_{N}\smallsetminus Z$ admits a geometric quotient $(U\times \bP_{N}\smallsetminus Z)\mmod T$ and
    	\item the geometric quotient $(U\times \bP_{N}\smallsetminus Z)\mmod T$ is locally projective over $U\mmod T$.
   \end{enumerate}  
\end{lem}
\begin{proof}
We already noted above that 
the $\cO(d)$-unstable points for the $\bG_m$-action on $\bP(V_{N,\pm})$ are exactly the fixed points which have codimension $N$ and all non-fixed points are stable.

Therefore the $\cO(d,\dots,d)$-unstable points on $\bP_N=\prod_{i=1}^r \bP(V_{N,\pm})$ are also exactly the points for which at least one coordinate is a fixed point. These are again of codimension $N$.

As by the Hilbert-Mumford criterion stability is unchanged under small changes of the linearization, the same holds for $\cO(d,\dots,d)_\chi$ whenever $\chi$ is a character of $T$ that is sufficiently small with respect to $d$.

A point in $U\times \bP_{N}$ is $\cL_{U,\chi}^d$-unstable if and only if the Hilbert-Mumford weight is $\geq 0$ for some cocharacter $\lambda\colon \bG_m\to T=\bG_m^r$. As the line bundle $\cL_{U}^d$ is obtained as pull-back from $\bP_{N}$, the weight is given by the weight of the image of the point under the projection to $\bP_{N}$. In particular all unstable points have to map to unstable points in $\bP_N$, so the codimension of the unstable points is at least $N$ and the projection from the semistable locus $(U\times \bP_{N})^{sst} \to U$ is surjective. 

Next we claim that for a suitable choice of $d,\chi$ semistability and stability agree on $U\times \bP_N$. If not, there exists a point $(u,p)\in U\times \bP_{N}$ that is fixed by some cocharacter $\lambda\colon \bG_m \to \Stab_{T}(u,p)$, in particular $\Stab_{T}(u)$ is non-trivial, but contained in the kernel of the morphism $T \to \bG_m$ defined by the linearization of $\cL^d_\chi$ at $p$. As $U$ is of finite type, only finitely many subgroups of $T$ appear as stabilizers of points of $U\times \bP_N$. Therefore there exists a character $\chi\in X^*(T)$ such that $\Stab_{T}(u,p)$ is not contained in the kernel.  	

The last claim on local projectivity of the quotient is local on $U\mmod T$. Since $U$ is normal, it has a $T$-invariant affine open cover by Sumihiro's theorem and thus the quotient $U\mmod T$ is a scheme. So we can assume that $U\mmod T$ is affine and as $[U/T]\to U\mmod T$ is cohomologically affine this implies that $U$ is affine. In this case the Hilbert-Mumford criterion determines GIT-stability on $U\times \bP_{N}$ and therefore $(U\times \bP_{N})^{\sst}\mmod T \to U\mmod T$ is projective.
\end{proof}

\begin{proof}[Proof of \Cref{Prop:semisimplepure}]
We will show that for any integer $c>0$ the truncated complex $\tau^{\leq c}\bR p_*\ubK_{[U/T]}[\dim [U/T]]$ coincides with the truncation $\tau^{\leq c} K_c$ of a semisimple pure complex $K_c$ on $U\mmod T$ that contains $\IC_{U\mmod T}$ as a direct summand. As the cohomological dimension of the finite type scheme $U\mmod T$ is finite, this implies that $\bR p_*\ubK_{[U/T]}[\dim [U/T]]$ is itself pure and semisimple.

To construct such a complex $K_c$ for any $c>0$, we apply \Cref{lem:ProperCover} to obtain a cover $p_U\colon [U \times \bP_N/T] \to [U/T]$ that is representable, smooth and projective, so that $K:=\bR p_{U,*} \ubK_{[U\times \bP_N/T]}[\dim [U/T]]$ is a semisimple complex that contains the constant sheaf $\ubK_{[U/T]}[\dim [U/T]]$ as a direct summand. 

It therefore suffices to prove the claim for this complex $K$. Now by \Cref{lem:ProperCover} we also know that the unstable part $Z\subset U \times \bP_N$ has codimension $\geq c$, so that for the inclusion $j\colon U \times \bP_N\smallsetminus Z \to U \times \bP_N$ the cone of
$$\ubK_{[U \times \bP_N/T]} \to \bR j_* \ubK_{[(U \times \bP_N\smallsetminus Z)/T]}$$ is concentrated in cohomological degree $\geq 2c-1$. Thus the induced morphism 
$$\bR p_* K \to \bR p_* \bR p_{U,*} \bR j_*\ubK_{[(U \times \bP_N\smallsetminus Z)/T]}[\dim [U/T]] =:K_c$$  is an isomorphism in cohomological degrees $\leq 2c-2-\dim[U/T]$. 

Now the composition 
$p\circ p_U\circ j$ factors as
$$[(U \times \bP_N\smallsetminus Z)/T] \map{q} (U \times \bP_N\smallsetminus Z)\mmod T \map{\overline{p}} U\mmod T.$$  
The first map is a geometric quotient and therefore $$\bR q_* \ubK_{[(U \times \bP_N\smallsetminus Z)/T]}[\dim(U\times\bP_N/T)] =\ubK_{(U \times \bP_N\smallsetminus Z)\mmod T}[\dim[U\times \bP_N/T]]$$ and this complex is self-dual. The second morphism is locally projective by construction, so the complex $$K_c[\dim \bP_N]=\bR \overline{p}_* \ubK_{(U \times \bP_N\smallsetminus Z)\mmod T]}[\dim [U\times \bP_N/T]]$$ is again a pure self-dual complex. By the decomposition theorem \cite{BBDG} this complex is semisimple. As $p_*\ubK_{[U/T]}=\ubK_{U\mmod T}$ the complex $K_c$ does contain $\IC_{U\mmod T}$ as a direct summand.
\end{proof}

\end{document}